\documentclass[11pt,leqno,twoside]{amsart}
\usepackage{amssymb,amsmath,amsthm,soul,color}
\usepackage{t1enc}
\usepackage[cp1250]{inputenc}
\usepackage{a4,indentfirst,latexsym}
\usepackage{graphics}
\usepackage{mathrsfs}
\usepackage{cite,enumitem,graphicx}
\usepackage[colorlinks=true,urlcolor=blue,
citecolor=red,linkcolor=blue,linktocpage,pdfpagelabels,
bookmarksnumbered,bookmarksopen]{hyperref}
\usepackage[english]{babel}
\usepackage[left=2.61cm,right=2.61cm,top=2.72cm,bottom=2.72cm]{geometry}
\usepackage[metapost]{mfpic}
\usepackage[hyperpageref]{backref}
\usepackage[colorinlistoftodos]{todonotes}
\usepackage[normalem]{ulem}

\makeatletter
\providecommand\@dotsep{5}
\def\listtodoname{List of Todos}
\def\listoftodos{\@starttoc{tdo}\listtodoname}
\makeatother

\numberwithin{equation}{section}
\newtheorem{Th}{Theorem}[section]
\newtheorem{Prop}[Th]{Proposition}
\newtheorem{Lem}[Th]{Lemma}

\newtheorem{Rem}[Th]{Remark}

   \newcommand{\eps}{\varepsilon}

   \def\ba{{\bf A}}

   \def\Z{\mathbb{Z}}

   \def\R{\mathbb{R}}

   \def\J{\mathcal{J}}
   
   \def\cN{\mathcal{N}}

   \def\D{\mathcal{D}}

   \def\RT{\mathbb{R}^3}
   \def\n{\nabla}
	\def\de{\partial}

\newcommand{\cM}{{\mathcal M}}

\newcommand{\cC}{{\mathcal C}}

\title[Ground states for KGMP systems]{Vortex ground states for Klein-Gordon-Maxwell-Proca type systems}

\author[P. d'Avenia]{Pietro d'Avenia}
\author[J. Mederski]{Jaros\l aw Mederski}
\author[A. Pomponio]{Alessio Pomponio}

\address[P. d'Avenia and A. Pomponio]{\newline\indent
Dipartimento di Meccanica, Matematica e Management
\newline\indent 
Politecnico di Bari
\newline\indent
Via Orabona 4,  70125  Bari, Italy}
\email{\href{mailto:pietro.davenia@poliba.it}{pietro.davenia@poliba.it}}
\email{\href{mailto:alessio.pomponio@poliba.it}{alessio.pomponio@poliba.it}}

\address[J. Mederski]{\newline\indent 
Faculty of Mathematics and Computer Science
\newline\indent 
Nicolaus Copernicus University
\newline\indent
ul. Chopina 12/18, 87-100 Toru\'n}
\email{\href{mailto:jmederski@mat.umk.pl}{jmederski@mat.umk.pl}}

\thanks{P. d'Avenia and A. Pomponio are partially supported by  a grant of the group GNAMPA of INdAM. J.M. was partially supported by the National Science Centre, Poland (Grant No. 2013/09/B/ST1/01963 and Grant No. 2014/15/D/ST1/03638)}

\subjclass[2010]{35Q60, 35J47, 35J50, 35A01.}
\date{\today}
\keywords{Elliptic systems, Klein-Gordon-Maxwell-Proca system, Klein-Gordon-Maxwell system, ground states, vortex solutions}

\begin{document}
\begin{abstract} We look for
three dimensional vortex-solutions, which have finite energy and are stationary solutions of Klein-Gordon-Maxwell-Proca type systems of equations. We prove the existence of three dimensional cylindrically symmetric vortex-solutions having a least possible energy among all symmetric solutions. Moreover we  show that if the Proca mass disappears then the solutions tends to a solution of the Klein-Gordon-Maxwell system.
\end{abstract}

\maketitle

\section{Introduction}

In the last years a wide literature has been devoted to gauge field theories. In a gauge theory the action is invariant under a continuous symmetry group that depends on spacetime. Gauge theories are used to describe theoretically fundamental forces of nature as electromagnetism, weak force and strong force and their interaction with matter field. These models are also verified experimentally with a high precision. Klein-Gordon-Maxwell type  systems, in which we are interested, fall within this framework.  In this case, the starting point is the nonlinear Klein-Gordon Lagrangian density
\begin{equation}
\label{eq:LKG}
\mathcal{L}_{\rm KG}(\psi)
=\frac{1}{2}[|\partial_t\psi|^2-|\nabla\psi|^2-m^2|\psi|^2]+F(x,\psi),
\end{equation}
where $(t,x)\in \mathbb{R}\times\mathbb{R}^3$, $\psi:\mathbb{R}\times\mathbb{R}^3\to \mathbb{C}$ represents the matter field, $m$ is a constant (the mass of $\psi$) and $F:\mathbb{R}^3\times\mathbb{C} \to \mathbb{R}$ is a nonlinearity, which can represent the interaction among many particles, such that $F(x,0)=0$ and $F(x,\psi)=F(x,|\psi|)$.\\
To study the interaction of the field $\psi$ with its own electromagnetic field, whose gauge potential is given by $(\phi,{\bf A})$ ($\phi:\mathbb{R}\times\mathbb{R}^3\to \mathbb{R}$ and ${\bf A}:\mathbb{R}\times\mathbb{R}^3\to \mathbb{R}^3$ are respectively the electric and the magnetic potentials), it is usual to consider the gauge covariant derivatives instead of the classical ones, replacing in \eqref{eq:LKG} the time derivative $\partial_t$ and the spatial derivatives $\nabla$ respectively with $\partial_t +iq\phi$ and $\nabla - iq{\bf A}$, where $q$ is a coupling constant (see \cite{F,N1,N2}). Thus we get
\[
\tilde{\mathcal{L}}_{\rm KG}(\psi,\phi,{\bf A})
=\frac{1}{2}\left[|\partial_t\psi+iq\phi\psi|^2-|\nabla\psi- iq{\bf A}\psi|^2-m^2|\psi|^2\right]+F(x,\psi).
\]
To this Lagrangian density we have to add the electromagnetic one in the vacuum. In this paper we consider the Maxwell and Maxwell-Proca Lagrangian densities
\[
\mathcal{L}_{\rm MP}(\phi,{\bf A})
=\underbrace{\frac{1}{2}[|\partial_t{\bf A}+\nabla\phi|^2-|\nabla\times{\bf A}|^2 ]}_{\text{Maxwell Lagrangian density}} + \underbrace{\frac{\mu^2}{2} [|\phi|^2-|{\bf A}|^2]}_{\text{mass term}}.
\]
The Maxwell-Proca Lagrangian density is a {\em massive version} of the Maxwell one and coincides with it if $\mu=0$. This model was introduced by Alexandre Proca (see \cite{Proca1,Proca2,Proca3,Proca4,Proca5}) under the de Broglie influence as a generalization of the Maxwell's one.  The {\em mass term} is due to the effects of a photon rest mass (which can be small but still nonzero) and the invariance of electrodynamics under transformations of special relativity is preserved (see \cite{GN} and references therein for more details). 
Hence the total Lagrangian density is given by
\[
\mathcal{L}(\psi,\phi,{\bf A})=\tilde{\mathcal{L}}_{\rm KG}(\psi,\phi,{\bf A})+\mathcal{L}_{\rm MP}(\phi,{\bf A})
\]
and the Euler-Lagrange equations of the total action
\[
\mathcal{S}(\psi,\phi,{\bf A})=\iint \mathcal{L}(\psi,\phi,{\bf A})\ dxdt
\]
are
\begin{equation}
\label{eq:KGM1}
\begin{cases}
(\partial_t+iq\phi)^2\psi-(\nabla-iq{\bf A})^2\psi +m^2 \psi
= \partial_{\psi} F(x,\psi),\\
-\nabla\cdot(\partial_t{\bf A}+\nabla\phi) + \mu^2\phi
=- q[\mathfrak{Im}(\partial_t\psi\cdot\bar{\psi})+q\phi|\psi|^2],\\
\nabla\times\nabla\times {\bf A}+\mu^2{\bf A} + \partial_t(\partial_t{\bf A}+\nabla\phi)
=q[\mathfrak{Im}(\nabla\psi\cdot\bar{\psi})-q{\bf A}|\psi|^2],
\end{cases}
\qquad
\hbox{in }\mathbb{R}\times\mathbb{R}^3.
\end{equation}

A very interesting problem regarding system \eqref{eq:KGM1} is the existence of solitary wave solutions, namely solutions whose finite energy travels as a localized packet. These solutions are strictly joint with {\em solitons}, i.e. solitary waves that exhibit a very strong form of stability.

Starting from the pioneering paper \cite{BF2002}, many papers have been devoted to the study of this type of systems and, in particular, to the existence of stationary solutions of \eqref{eq:KGM1}, i.e. solution of the form $\psi(t,x)=u(x)e^{iS(t,x)}$, with $u:\mathbb{R}^3\to\mathbb{R}$ and $S:\mathbb{R}\times\mathbb{R}^3\to\mathbb{R}$, especially in the purely electrostatic case ($\phi=\phi(x)$ and ${\bf A}={\bf 0}$), with $\mu=0$. Existence and non-existence results, under different assumptions on the nonlinearity and on $m$ and $\omega$ are present in \cite{DM,DM2,APP}, while the existence of a ground state has been considered in \cite{AP,Wang}. 
Moreover we mention \cite{DPS09,DPS10}, for the bounded domain case, and \cite{CGM,DHV,GMP,HT,HW}, for the case $\ba= {\bf 0}$ and $\mu\neq 0$ on manifolds. Finally Klein-Gordon equations coupled with Born-Infeld type equations have been treated in \cite{BDP,DP,yu}.

In this paper, instead, we are interested in the electromagnetostatic case (nontrivial $\phi=\phi(x)$ and ${\bf A}={\bf A}(x)$) which has been studied in a smaller number of articles (see \cite{BBS,BF2009NA,BF2010CMP}), only for the Klein-Gordon-Maxwell system ($\mu=0$). In particular, from the physical point of view, it is important to find ground state solutions which are minimizers of the energy functional among a set of nontrivial solutions and, in this work, we find symmetric ground states for any $\mu\in\R$. Our results are also new when $\mu=0$. Besides we investigate the behavior of solutions when the Proca mass $\mu$ disappears. 

These lists of literature, of course, are not complete and we refer to \cite{BFbook} and references therein for a more exhaustive description. 

When we look for stationary solutions, the total action $\mathcal{S}$ depends on the unknowns $u,S,\phi,{\bf A}$ and its Euler-Lagrange equations, in the particular case $S(t,x)=S_0(x)-\omega t$ with $\omega\in\mathbb{R}$, $\phi=\phi(x)$ and ${\bf A}={\bf A}(x)$, are
\begin{equation}
\label{eq:KGM2}
\begin{cases}
-\Delta u + [m^2-(\omega-q\phi)^2] u + |\nabla S_0-q{\bf A}|^2 u
= \partial_u F(x,u),\\
\nabla\cdot[(\nabla S_0-q{\bf A})u^2] =0,\\
-\Delta\phi + \mu^2\phi
=q(\omega-q\phi)u^2,\\
\nabla\times\nabla\times {\bf A} +\mu^2{\bf A}
=q(\nabla S_0-q{\bf A})u^2,
\end{cases}
\qquad
\hbox{in }\mathbb{R}^3.
\end{equation}

Observe that,  for instance, considering $\n S_0(x)\neq {\bf 0}$ a.e. in $\RT$, we get that if $\ba={\bf 0}$, then the unique solution of \eqref{eq:KGM2} is the trivial one and so the purely electrostatic case is meaningless.

If $\mu=0$, moreover, the second equation in \eqref{eq:KGM2} follows from the fourth one. If $\mu\neq 0$, second and fourth equation in \eqref{eq:KGM2} imply $\operatorname{div} {\bf A} =  0$. Thus, while in the Klein-Gordon-Maxwell system ($\mu=0$) the {\em natural constraint} $\operatorname{div} {\bf A} = 0$ seems to be a technical requirement in order to avoid the problems related to the curl-curl operator in the last equation of \eqref{eq:KGM2} (see e.g. \cite{BFbook}), in the Klein-Gordon-Maxwell-Proca case ($\mu\neq 0$) we need to require this property in order to get solutions. Thus, system \eqref{eq:KGM2} is equivalent to
\begin{equation}
\label{eq:KGM3}
\begin{cases}
-\Delta u + [m^2-(\omega-q\phi)^2] u + |\nabla S_0-q{\bf A}|^2 u
= \partial_u F(x,u),\\
-\Delta\phi + \mu^2\phi
=q(\omega-q\phi)u^2,\\
\nabla\times\nabla\times {\bf A} +\mu^2{\bf A}
=q(\nabla S_0-q{\bf A})u^2,
\end{cases}
\qquad
\hbox{in }\mathbb{R}^3,
\end{equation}
with the additional constraint that $\operatorname{div}{\bf A}=0$ if $\mu\neq 0$. We will show (see Section \ref{se:vs}) that, including such a condition in our {\em setting}, we get a {\em natural constraint} and for this reason we can restrict to consider system \eqref{eq:KGM3} for all $\mu \in \R$.

Hence, as e.g. in \cite{BF2009NA,BF2010CMP},  we set
\begin{equation}\label{Sigma}
\Sigma=\left\{(x_1,x_2,x_3)\in \mathbb{R}^3 : x_1=x_2=0 \right\}
\end{equation}
and we consider the maps
\begin{equation}
\label{theta}
\theta:\mathbb{R}^3\setminus \Sigma \to \mathbb{R}/(2\pi\mathbb{Z}),
\qquad
\theta(x_1,x_2,x_3)=\mathfrak{Im}\log (x_1+ix_2),
\end{equation}
and $S_0(x)=\ell\theta(x)$, with $\ell\in\mathbb{Z}\setminus\{0\}$. Thus we look for vortex solutions $\psi(t,x)=e^{i(\ell\theta(x)-\omega t)} u(x)$ of \eqref{eq:KGM1} and so $(u,\phi,{\bf A})$'s that solve
\begin{equation}
\label{P}
\tag{$\mathcal{S}_{\mu}$}
\begin{cases}
-\Delta u + [m^2-(\omega-q\phi)^2] u + |\ell \nabla\theta-q{\bf A}|^2 u= f(x,u),\\
-\Delta\phi + \mu^2\phi =q(\omega-q\phi)u^2,\\
\nabla\times\nabla\times {\bf A}  +\mu^2{\bf A} =q(\ell \nabla\theta-q{\bf A})u^2,
\end{cases}
\qquad
\hbox{in }\mathbb{R}^3,
\end{equation} 
with the additional condition $\operatorname{div}{\bf A}=0$ for $\mu\neq 0$. In the following we refer to \eqref{P} as Klein-Gordon-Maxwell-Proca system if $\mu\neq 0$, otherwise, as Klein-Gordon-Maxwell.

%
%

In particular, in this paper we are interested in the existence of {\em symmetric ground state} solutions of \eqref{P}, namely solutions that minimize the energy on suitable Nehari-type {\em manifolds} (see Section \ref{section:Th1} and Section \ref{se:MP} for more details): this kind of solutions are often a good starting point in order to get solitons.
We emphasize that, with respect to previous papers on this topic, we prove our existence results for an arbitrary $\mu\in\R$, under general weak assumptions on the nonlinearity $f$, without any constraint on the $L^2$-norm of $u$ ({\em the total charge}) and with a nontrivial magnetostatic potential $\ba$. Moreover, up to our knowledge, another novelty is the study of the behaviour of these ground states as $\mu\to 0$ according to the physical meaning of such a parameter.

Let $0<\omega^2<m^2$ and
assume that the function $f:\mathbb{R}^3\times\mathbb{R} \to \mathbb{R}$ satisfies:
\begin{enumerate}[label=(F\arabic*),ref=F\arabic*]
	\item \label{SW1} 
	for any $s\in\R$, $f(\cdot,s)$ is measurable, depends on $(r,x_3)$ and is $\Z$-periodic with respect to $x_3$, $f(x,\cdot)$ continuous for a.e. $x\in\R^3$ and
 there are $a>0$ and $2<p<6=2^*$ such that
	$$|f (x,s)|\leq a(1 + |s|^{p-1}), \quad
	\hbox{ for a.e. $x\in\R^3$ and all $s\in\R$} ;$$ 
	\item \label{SW2} $f(x,s)= o(s)$ uniformly in $x$ as $|s|\to0$;
	\item \label{SW3}  $F(x,s)>0$ for all $s\neq 0$ and a.e. $x\in\R^3$, where $F(x,s)=\int_0^s f(x,t) dt$;
%
%
\item \label{SW4} $f(x,\cdot)$ is $\cC^1$ for a.e. $x\in\R^3$ and there is $\sigma>2$ such that for all $s\in\R$ and a.e. $x\in\R^3$
$$(\sigma-1)f(x,s)s \leq \partial_s f(x,s)s^2.$$
\end{enumerate}

Our first main result is the following
\begin{Th}\label{ThMain1}
Let
\begin{equation}
\label{condomegamsigma}
\begin{cases}
m>\omega, & \hbox{for } \sigma \geq 4,\\
\displaystyle(\sigma-2)m^2 - \frac{\sigma^2 - 4 \sigma +8}{4}\omega^2>0, & \hbox{for } 2< \sigma < 4.
\end{cases}
\end{equation}
If \eqref{SW1}-\eqref{SW4} hold, then
for every $\mu\in\R$, the system \eqref{P} admits a symmetric finite energy ground state solution $(u_{\mu},\phi_{\mu}, {\bf A}_{\mu})$.
\end{Th}

Observe that (\ref{SW4}) implies that $s\mapsto f (x,s)/|s|^{\sigma-1}$ is nondecreasing on $(-\infty, 0)$ and $(0,+\infty)$.

If $\sigma\geq 4$ then \eqref{SW4} can be weakened, in the spirit of \cite{SzulkinWeth}, and we consider the following monotonicity condition:
\begin{enumerate}[label=(F\arabic*),ref=F\arabic*]\setcounter{enumi}{4}
	\item \label{SW5} $u\mapsto f (x,u)/|u|^3$ is nondecreasing on $(-\infty, 0)$ and $(0,+\infty)$.
\end{enumerate}
With this weaker assumption we have

\begin{Th}\label{ThMain2}
If \eqref{SW1}--\eqref{SW3} and \eqref{SW5} hold, then
for every $\mu\in\mathbb{R}$, the system \eqref{P} admits a symmetric finite energy ground state solution $(u_{\mu},\phi_{\mu}, {\bf A}_{\mu})$.
\end{Th}

\begin{Rem}
Observe that both (\ref{SW4}) and (\ref{SW5}) imply $f(\cdot,s)=o(s)$ a.e. in $\mathbb{R}^3$ (we have respectively that $f(\cdot,s)=O(s^{\sigma-1})$ and $f(\cdot,s)=O(s^3)$). Thus, if we have some uniformity with respect to $x\in\mathbb{R}^3$, e.g. if $f$ does not depend on $x$, (\ref{SW2}) can be deleted.
Moreover \eqref{SW1}--\eqref{SW3} and \eqref{SW4}, respectively \eqref{SW5}, imply that $p\geq \sigma$, respectively $p\geq 4$ (see Lemma \ref{Le29}). Therefore, while in Theorem \ref{ThMain1} we can consider also $p\in (2,4)$, in Theorem \ref{ThMain2} we need to require that $p\geq 4$. On the other hand, the assumptions on the nonlinearity $f$ in Theorem \ref{ThMain2}  are weaker than in Theorem \ref{ThMain1}. 
\end{Rem}



In this kind of problem the classical approach consists in considering a functional which depends only on two variables, $(u,\ba)$,  since the second equations of \eqref{P} is uniquely solved, fixed $u$, (see for example \cite{BFbook} and the references therein).
In the assumptions of Theorem \ref{ThMain1}, the minimization on the Nehari manifold of the two variable functional seems to fail, because it is not clear if the Nehari manifold is a natural constraint. We overcome this difficulty by the following trick (that, up to our knowledge, is used here for the first time): we consider also the unique symmetric solution of the third equation of \eqref{P}, for $u$ fixed, and then we minimize on the {\em symmetric} Nehari manifold of a one variable functional.
The classical approach, instead, works in the hypotheses of Theorem \ref{ThMain2}: even if we cannot minimize directly on the Nehari {\em manifold} of the two variable functional due to the general assumptions on $f$, in particular the lack of an Ambrosetti-Rabinowitz condition, we are able to find a Mountain Pass solution and to show that it is actually a minimizer on the Nehari {\em manifold}.

We emphasize, moreover, that our ground states are also least energy solutions among all cylindrically symmetric ones.

%

We conclude the paper  analyzing the behavior of the symmetric ground state solutions found in the previous theorems as $\mu$ goes to zero. We show, in particular, that these solutions tend to a weak solution of the Klein-Gordon-Maxwell system \eqref{P} with $\mu=0$; however we do not know if this weak solution is again a ground state (see Section \ref{se:mu} for more details).

\begin{Th}\label{ThMain3}
Suppose that \eqref{SW1}-\eqref{SW3} hold and, additionally, assume \eqref{SW4} or \eqref{SW5}. 
If $(u_{\mu},\phi_{\mu}, {\bf A}_{\mu})$ is a symmetric finite energy ground state solution of \eqref{P} and $\mu\to 0$, then $(u_{\mu},\phi_{\mu}, {\bf A}_{\mu})$ tends weakly in $H^1(\R^3)\times \D^{1,2}(\R^3)\times (\D^{1,2}(\R^3))^3$ and a.e. in $\R^3$ to a symmetric weak solution $(u_{0},\phi_{0}, {\bf A}_{0})$ of the Klein-Gordon-Maxwell system \eqref{P} with $\mu=0$.
\end{Th}

Observe finally that, with slight modifications, we could consider the case $\omega=0$. In this situation we are in the purely magnetostatic case, namely $\phi=0$.

In the following we denote by $C,C_i$ positive constants which can change from line to line. If not specified, all the integrals are evaluated on the whole $\R^3$.

\section{Variational setting}\label{se:vs}
In this section we give some properties on the variational structure of our problem and on symmetries we consider in order to get the desired solutions of \eqref{P}.

\subsection{Space and functional}

Let $H^1(\mathbb{R}^3)$ be the usual Sobolev space and $\D^{1,2}(\RT)$ be the completion of $C_0^\infty(\RT)$ with respect to the norm $\|\n \cdot\|_2$ and
\[
H_{\mu}=
\begin{cases}
H^{1}(\mathbb{R}^3) & \hbox{for }\mu\neq 0,\\
\D^{1,2}(\R^3) & \hbox{for }\mu = 0,
\end{cases}
\] 
equipped with the norm  $(\|\nabla \cdot \|_2^2 + \mu^2 \|\cdot\|_2^2)^{1/2}$.
Let, moreover, 
$\hat{H}^1$ be the closure of $C_0^\infty(\mathbb{R}^3\setminus\Sigma)$ with respect to the norm
\[
\|u\| = \Big(\|\nabla u\|_2^2+(m^2-\omega^2)\|u\|_2^2+ \ell^2 \int \frac{u^2}{r^2}\,dx\Big)^{1/2},
\qquad
r=\sqrt{x_1^2+x_2^2},
\]
where $x=(x_1,x_2,x_3)\in \RT$.
Of course $\hat{H}^1\subset H^1(\mathbb{R}^3)$ and it can be proved that $\hat{H}^1\cap C_0^\infty(\mathbb{R}^3)$ is dense in $\hat{H}^1$.

%
On $\hat{H}^1\times H_\mu \times (H_\mu)^3$ we define the functional
\begin{align*}
I(u,\phi,{\bf A})
&:=
\frac{1}{2}\|\nabla u\|_2^2 +\frac{m^2}{2}\|u\|_2^2 -\frac{1}{2} \|\nabla \phi\|_2^2 - \frac{\mu^2}{2} \|\phi\|_2^2 + \frac{1}{2}\|\nabla\times{\bf A}\|_2^2 + \frac{\mu^2}{2} \|{\bf A}\|_2^2 \\
&\qquad 
+ \frac{1}{2} \int |\ell\nabla\theta - q {\bf A}|^2 u^2 - \frac{1}{2} \int (\omega-q\phi)^2 u^2 - \int F(x,u).
\end{align*}
As in \cite[Section 3.2]{BF2010CMP} we can prove the following
\begin{Prop}\label{PropDitribSol}
The functional $I$ is of class $\cC^1$ in $\hat{H}^1\times H_\mu \times (H_\mu)^3$ and its critical points that satisfy the additional condition $\operatorname{div}{\bf A}=0$ are solutions of \eqref{P} in the sense of distributions, i.e.
\[
\begin{cases}
\displaystyle 
\int \nabla u \cdot \nabla v 
+ \int [m^2-(\omega-q\phi)^2] u v 
+ \int |\ell \nabla\theta-q{\bf A}|^2 u v
= \int f(x,u) v
\quad \hbox{for all } v \in C_0^\infty(\mathbb{R}^3),\\
\displaystyle 
\int \nabla \phi \cdot \nabla w 
+ \mu^2 \int \phi w 
=q\int (\omega-q\phi)wu^2 
\quad \hbox{for all } w \in C_0^\infty(\mathbb{R}^3),\\
\displaystyle 
\int \nabla {\bf A} \cdot \nabla {\bf B}  
+\mu^2\int {\bf A}\cdot{\bf B} 
=q\int (\ell \nabla\theta-q{\bf A})u^2 \cdot{\bf B} 
\quad \hbox{for all }{\bf B} \in (C_0^\infty(\mathbb{R}^3))^3.
\end{cases}
\]
\end{Prop}

\subsection{Reduction argument on the second equation of \eqref{P}}\label{reductionargument}

As it is classical in this kind of problem, we note that, if we fix $u\in H^1(\mathbb{R}^3)$, the second equation of \eqref{P} admits a unique solution $\phi_u\in H_\mu$.
Such a solution can be obtained as the unique minimizer of the functional 
\[
K(\phi):=\frac{1}{2}\|\nabla\phi\|_2^2+ \frac{\mu^2}{2}\|\phi\|_2^2+\frac{1}{2} \int (\omega-q\phi)^2 u^2
\] 
defined on $H_\mu$.

In this section we want to introduce some useful properties of $\phi_u$.

\begin{Lem}\label{Lempropphi}
	The function $\Phi:=u\mapsto\phi_u$ is of class $\cC^1$ and for all $v\in H^{1}(\mathbb{R}^3)$,
	\begin{equation}
	\label{Phi'}
	\Phi'(u)[v]=-2q(\Delta-\mu^2-q^2u^2)^{-1} [(\omega-q\phi_u)uv].
	\end{equation}
\end{Lem}
\begin{proof}
	We proceed as in \cite[Proposition 2.1]{DM}. So we define the map $T: H^1(\mathbb{R}^3)\times H_\mu \to H_\mu$ such that 
	\[
	T(u,\phi)=q(\Delta-\mu^2)^{-1}[(\omega-q\phi)u^2] + \phi.
	\]
	Simple calculations show that
	\[
	\partial_u T(u,\phi)=v\in H^1(\mathbb{R}^3) \longmapsto 2q (\Delta - \mu^2)^{-1}[(\omega-q\phi)uv]
	\]
	and
	\[
	\partial_\phi T(u,\phi)=\psi\in H_\mu \longmapsto -q^2(\Delta - \mu^2)^{-1}[u^2\psi] + \psi.
	\]
	We have that $\partial_\phi T(u,\phi)$ is invertible and that
	\[
	(\partial_\phi T(u,\phi))^{-1}=(\Delta - \mu^2 - q^2u^2)^{-1}\circ(\Delta-\mu^2).
	\]
	Thus, by Implicit Function Theorem, we conclude.
\end{proof}
%
Moreover we have the following further useful properties. 
\begin{Lem}
	\label{le:phipsi}
Let $u\in H^1(\mathbb{R}^3)$. Then $\phi_u$ satisfies
\begin{equation}\label{bohh}
\|\nabla\phi_u\|_2^2 + \mu^2 \|\phi_u\|_2^2
=q\int(\omega-q\phi_u)\phi_u u^2,
\end{equation}
\begin{equation}
\label{bohhh}
0\leq \phi_u \leq \omega/q,
\end{equation}
\begin{equation}
\label{124}
\|\nabla\phi_u\|_2^2 + \mu^2 \|\phi_u\|_2^2
\leq \omega^2 \|u\|_2^2,
\end{equation}
and there exists $C>0$ such that
\begin{equation}
\label{eq:125}
\|\nabla\phi_u\|_2^2 + \mu^2 \|\phi_u\|_2^2\leq C \|u\|_{12/5}^4.
\end{equation}
Moreover, $\psi_u := \Phi'(u)[u]/2$ satisfies
\begin{align}
-\Delta\psi_u + \mu^2 \psi_u &=q(\omega-q\phi_u-q\psi_u)u^2, \label{cacca}
\\
\omega \int  \psi_u u^2 
&= \int (\omega-q\phi_u)\phi_u u^2, \label{cacca2}
\end{align}
\begin{equation}
\label{cc3}
0\leq \psi_u \leq \phi_u.
\end{equation}
\end{Lem}
\begin{proof}
Fix $u\in H^1(\mathbb{R}^3)$. 
Equation \eqref{bohh} is simply obtained    
multiplying the second equation of \eqref{P} by $\phi_u$ and integrating.
\\
Let us prove that $\phi_u \geq 0$. Assume by contradiction that it is not true. Then let $\Gamma=\left\{x\in\mathbb{R}^3\mid \phi_u \leq 0 \right\}$ and $(\phi_u)_-=\min\{\phi_u,0\}$. Multiplying by $(\phi_u)_-$ the second equation of \eqref{P} and integrating on $\Gamma$ we get
\[
\int_{\Gamma} |\nabla\phi_u|^2 + \mu^2 \int_{\Gamma} \phi_u^2
=q\omega\int_{\Gamma}\phi_u u^2 - q^2 \int_{\Gamma}\phi_u^2 u^2
\]
and we reach a contradiction since the right hand side is negative and the left hand side is positive. To prove that $\phi_u \leq \omega/q$, assume by contradiction that it is not true and let $\Gamma=\left\{x\in\mathbb{R}^3\mid \omega -q\phi_u \leq 0 \right\}$ and $(\omega - q\phi_u)_-=\min\{\omega - q\phi_u,0\}$. Multiplying by $(\omega - q\phi_u)_-$ the second equation of \eqref{P} and integrating on $\Gamma$ we get
\[
-q\int_{\Gamma} |\nabla\phi_u|^2 + \mu^2 \int_{\Gamma} \phi_u (\omega - q\phi_u)
=q\int_{\Gamma} (\omega-q\phi_u)^2 u^2
\]
and we reach again a contradiction since the left hand side is negative and the right hand side is positive.
Inequality \eqref{124} easily follows from \eqref{bohh} and \eqref{bohhh} and, again by \eqref{bohh}, 
	\begin{align*}
	\|\nabla\phi_u\|_2^2 + \mu^2 \|\phi_u\|_2^2
	\leq q\omega \int \phi_u u^2
	\leq q\omega \|\phi_u\|_6 \|u\|_{12/5}^2
	\leq C (\|\nabla\phi_u\|_2^2 + \mu^2 \|\phi_u\|_2^2)^{1/2} \|u\|_{12/5}^2
	\end{align*}
	so that \eqref{eq:125} holds.\\
By \eqref{Phi'} we have that
\[
\psi_u=-q (\Delta - \mu^2 - q^2u^2)^{-1} [(\omega-q\phi_u)u^2]
\]
and so \eqref{cacca} easily follows.\\
Moreover, since $\phi_u$ solves the second equation in \eqref{P} and by \eqref{cacca}, we have
\[
q\int (\omega - q \phi_u)\psi_u u^2
=\int (-\Delta \phi_u + \mu^2 \phi_u)\psi_u 
=\int(-\Delta\psi_u + \mu^2 \psi_u)\phi_u 
= q\int (\omega-q\phi_u-q\psi_u)\phi_u u^2
\]
and so \eqref{cacca2}.
\\
To prove \eqref{cc3}, we proceed as before. First, assume by contradiction that $\psi_u$ is not positive and let $\Gamma=\left\{x\in\mathbb{R}^3\mid \psi_u \leq 0 \right\}$ and $(\psi_u)_-=\min\{0,\psi_u\}$.
Multiplying \eqref{cacca} by $(\psi_u)_-$ we have that
\[
\int_\Gamma |\nabla \psi_u|^2 + \mu^2 \int_\Gamma |\psi_u|^2
= q\int_\Gamma (\omega-q\phi_u)\psi_u u^2-q\int_\Gamma |\psi_u|^2 u^2\leq 0
\]
and we reach a contradiction. Then, if by contradiction $\Gamma=\left\{x\in\mathbb{R}^3\mid \phi_u \leq \psi_u \right\}\neq\emptyset$, we consider $(\psi_u-\phi_u)_+=\max\{0,\psi_u-\phi_u\}$.
Since, by \eqref{cacca} and the second equation in \eqref{P},
\[
-\Delta(\psi_u - \phi_u) + \mu^2 (\psi_u-\phi_u) = -q^2\psi_u u^2,
\]
multiplying by $(\psi_u-\phi_u)_+$ we have that
\[
\int_\Gamma |\nabla (\psi_u - \phi_u)|^2 + \mu^2 \int_\Gamma |(\psi_u - \phi_u)|^2
= -q^2\int_\Gamma (\psi_u - \phi_u)\psi_u u^2\leq 0
\]
and we reach a contradiction.
\end{proof}

So we can introduce the {\em reduced} functional
\begin{equation}
\label{defJ}
\begin{split}
J(u,{\bf A})
&=I(u,\phi_u,{\bf A})\\
&=
\frac{1}{2} \|\nabla u\|_2^2
+\frac{m^2-\omega^2}{2}\|u\|_2^2
+\frac{q\omega}{2}\int \phi_u u^2
+ \frac{1}{2}\|\nabla\times{\bf A}\|_2^2
+\frac{\mu^2}{2}\|{\bf A}\|_2^2\\
&\quad
+ \frac{1}{2} \int |\ell\nabla\theta- q{\bf A}|^2 u^2
- \int F(x,u)
\end{split}
\end{equation}
on $\hat{H}^1 \times (H_\mu)^3$. 
We have that $J$ is of class $\cC^1$ and if $(u,{\bf A})$ is a critical point of $J$, then $(u,\phi_u,{\bf A})$ is a critical point of $I$.


\subsection{Symmetries and natural constraint}

If $\mu\neq 0$ we could prove that, for fixed $u\in \hat{H}^1$, there exists a unique ${\bf A}_u$ which solves the third equation of \eqref{P}, minimizing the functional 
\begin{equation}\label{KAPPA}
{\bf K}({\bf A})
:=
\frac{1}{2}\|\nabla\times{\bf A}\|_2^2
+ \frac{\mu^2}{2}\|{\bf A}\|_2^2
+ \frac{q^2}{2}\int |{\bf A}|^2 u^2
- q\ell \int \nabla\theta\cdot{\bf A} u^2 
\end{equation}
defined in the Hilbert space $H(\operatorname{curl},\mu,u)$,
the completion of $(C_0^{\infty}(\R^3))^3$ with respect to the following norm
\[
\|{\bf A}\|_{\operatorname{curl},\mu,u}^2:= \|\nabla\times{\bf A}\|_2^2 +\mu^2\|{\bf A}\|_2^2+\|u{\bf A}\|_2^2,
\]
being ${\bf K}$ strictly convex and coercive on $H(\operatorname{curl},\mu,u)$.
However, in such a way, we do not get any information on $\operatorname{div}{\bf A}_u$ but, as observed in the Introduction, we are looking for solutions $(u,\phi, \ba)$, with $\operatorname{div}{\bf A}=0$.
Moreover, if $\mu= 0$, then $\|{\bf A}\|_{\operatorname{curl},0,u}$ does not define a norm, in general.

Hence, to avoid these difficulties, we consider the following symmetric setting similarly as in \cite{BF2009NA,BF2010CMP}, where the case $\mu=0$ has been considered for different types of nonlinearities.


Let us consider functions that are cylindrically symmetric, namely which depend only on the cylindrical coordinates $(r,x_3)$ and let us denote with $(C_0^\infty(\RT))_\sharp$, respectively with $(C_0^\infty(\RT\setminus \Sigma))_\sharp$, the subspace of cylindrically symmetric test functions in $\RT$, respectively in $\RT\setminus\Sigma$  ($\Sigma$ is defined in \eqref{Sigma}). Moreover let $H^1_\sharp(\mathbb{R}^3)$ be the closure of $(C_0^\infty(\RT))_\sharp$ with respect to the $H^1(\mathbb{R}^3)$-norm and,  analogously, we define $(H_\mu)_\sharp$ and $\hat{H}^1_\sharp$.\\
Observe that if $u\in H^1_\sharp(\mathbb{R}^3)$, then $\phi_u\in (H_\mu)_\sharp$.
Moreover, since we are looking for solutions with $\operatorname{div}{\bf A}=0$, we consider the set
\[
\mathcal{A}^\infty_0
=
\left\{
{\bf B}\in C_0^\infty(\mathbb{R}^3\setminus\Sigma,\mathbb{R}^3) \mid {\bf B}=b(r,x_3)\nabla\theta, b\in  C_0^\infty(\mathbb{R}^3\setminus\Sigma,\mathbb{R})
\right\},
\]
where $\theta$ is defined in \eqref{theta} and
\[
\nabla\theta(x)=\left(\frac{x_2}{r^2},-\frac{x_1}{r^2},0\right),
\]
and then the completion $\mathcal{A}$ of $\mathcal{A}^\infty_0$ with respect to the $(H_\mu)^3$-norm.\\
As in \cite[Lemma 15]{BF2010CMP} we can prove the following
\begin{Lem}\label{LemA}
For every ${\bf A}\in \mathcal{A}$ we have that $\operatorname{div}{\bf A}=0$, $\|\nabla\times{\bf A}\|_2=\| \nabla {\bf A}\|_2$, and so $\nabla\times\nabla\times {\bf A}= - \Delta {\bf A}$.
\end{Lem}
Thus we consider
\[
V:=\hat{H}^1_\sharp \times \mathcal{A}
\]
equipped with the product norm.
Note that $\mathcal{A}\subset (H_{\mu})^3$. 

In the next lemma we prove that  $V$ is a natural constraint and so we can reduce to look for critical points of $J$ on $V$. Note that Lemma \ref{LemA} implies that the critical points on $V$ satisfy $\operatorname{div}{\bf A}=0$, which is the additional condition that arises when we pass from system \eqref{eq:KGM2} to \eqref{P} in the Klein-Gordon-Maxwell-Proca case ($\mu\neq 0$). Therefore if $(u,{\bf A})\in V$ is a critical point of $J$, then 
$(u,\phi_u,{\bf A})$ is a solution of \eqref{P} for an arbitrary (fixed) value of $\mu\in\R$.
\begin{Lem}\label{eq:lemNaturalConstr}
Let $(u,{\bf A})\in V$. If $\partial_u J(u,{\bf A})[v]=0$ for any $v\in \hat{H}^1_\sharp$, then $\partial_u J(u,{\bf A})=0$ and if $\partial_{\bf A} J(u,{\bf A})[{\bf B}]=0$ for any ${\bf B}\in\mathcal{A}$, then $\partial_{\bf A} J(u,{\bf A})=0$.
\end{Lem}
\begin{proof}
We argue similarly as in \cite[Theorem 16]{BF2009NA}.
Suppose that $\partial_u J(u,{\bf A})[v]=0$ for any $v\in \hat{H}^1_\sharp$
 and let
$$\eta=-\Delta u + [m^2-(\omega-q\phi_u)^2] u + |\ell \nabla \theta-q{\bf A}|^2u-f(x,u).$$
Take any $v \in \hat{H}^1$ and let $v=v_1+v_2$,  where $v_1\in \hat{H}^1_{\sharp}$ and $v_2\in (\hat{H}^1_{\sharp})^{\perp}$.
Then
$$\partial_u J(u,{\bf A})[v]=\langle\eta,v\rangle=\partial_u J(u,{\bf A})[v_1]+\langle\eta,v_2\rangle=\langle\eta,v_2\rangle.$$
Since $u$, $\phi_u$, $|\ell \nabla \theta-q{\bf A}|^2$, and $f$ are cylindrically symmetric,
then by the density argument $\eta\in (\hat{H}^1_{\sharp})'$ and 
$$\partial_u J(u,{\bf A})[v]=0.$$ 
Similarly we suppose that $\partial_{\bf A} J(u,{\bf A})[{\bf B}]=0$ for any ${\bf B}\in \mathcal{A}$ and let
$$\xi=-\Delta {\bf A} + \mu^2 {\bf A} - q(\ell \nabla \theta-q{\bf A})u^2.$$
Take any ${\bf B} \in (H_{\mu})^3$ and let ${\bf B}={\bf B}_1+{\bf B}_2$,  where ${\bf B}_1\in \mathcal{A}$ and ${\bf B}_2\in \mathcal{A}^{\perp}$.
Then
$$\partial_{\bf A} J(u,{\bf A})[{\bf B}]=\langle\xi,{\bf B}\rangle=\partial_{\bf A} J(u,{\bf A})[{\bf B}_1]+\langle\xi,{\bf B}_2\rangle=\langle\xi,{\bf B}_2\rangle.$$
Similarly as in \cite[Lemma 12]{BF2009NA},  by the density argument $\xi\in \mathcal{A}'$ and 
$$\partial_{\bf A} J(u,{\bf A})[{\bf B}]=0.$$
\end{proof}



In this symmetric setting, we are able to prove the following result that holds true both for $\mu\neq 0$ and for $\mu=0$. 
\begin{Lem}\label{le:A}
For every $u\in \hat{H}^1_\sharp$ there exists a unique ${\bf A}_u\in\mathcal{A}$ that solves the third equation of \eqref{P}. 
Moreover we have
\begin{equation}
\label{Aueq}
\|\nabla\times{\bf A}_u\|_2^2 + \mu^2 \|{\bf A}_u\|_2^2
= q\int (\ell\nabla\theta- q{\bf A}_u)\cdot {\bf A}_u u^2,
\end{equation}
\begin{equation}
\|\nabla\times{\bf A}_u\|_2^2 + \mu^2 \|{\bf A}_u\|_2^2 +\int |\ell\nabla\theta- q{\bf A}_u|^2 u^2
=\ell^2 \int \frac{u^2}{r^2} - \ell q \int \nabla\theta\cdot {\bf A}_u u^2, \label{comevuoi}
\end{equation}
and
\begin{equation}
\label{Auineq}
0\leq q^2 \int |{\bf A}_u|^2 u^2 \leq \ell q \int \nabla\theta\cdot {\bf A}_u u^2\leq \ell^2 \int \frac{u^2}{r^2}.
\end{equation}
\end{Lem}
\begin{proof} Fix $u\in \hat{H}^1_\sharp$. 
Let us consider the functional ${\bf K}$, defined in \eqref{KAPPA}, on $\mathcal{A}$. It is strictly convex and coercive, hence there is a unique critical point, ${\bf A}_u$, of ${\bf K}$. In view of Lemma \ref{eq:lemNaturalConstr} we get that ${\bf A}_u$ solves the third equation of \eqref{P}.
Equation \eqref{Aueq} and the first two inequalities in \eqref{Auineq} are trivial. To conclude, it is enough to observe that
\begin{align*}
0
&\leq
\|\nabla\times{\bf A}_u\|_2^2 + \mu^2 \|{\bf A}_u\|_2^2 +\int |\ell\nabla\theta- q{\bf A}_u|^2 u^2\\
&=
q\int (\ell\nabla\theta- q{\bf A}_u)\cdot {\bf A}_u u^2 +\int |\ell\nabla\theta- q{\bf A}_u|^2 u^2\\
&=
\ell^2 \int \frac{u^2}{r^2} - \ell q \int \nabla\theta\cdot {\bf A}_u u^2.
\end{align*}
\end{proof}

The following regularity result is useful in order to deal with a further reduced functional.
\begin{Lem}\label{le:A'}
The map $\mathscr{A}:=u\in \hat{H}^1_\sharp \mapsto {\bf A}_u\in\mathcal{A}$ is of class $\cC^1$ and for all $v\in \hat{H}^1_\sharp$,
	\[
	\mathscr{A}'(u)[v]=-2q(\Delta-\mu^2-q^2u^2)^{-1} [(\ell \n \theta-q\ba_u)uv].
	\]
Moreover, $\Psi_u := \mathscr{A}'(u)[u]/2$ satisfies
\begin{equation}
-\Delta\Psi_u + \mu^2 \Psi_u =q(\ell \n \theta-q\ba_u-q\Psi_u)u^2, \label{caccaa}
\end{equation}
and so
\begin{equation}\label{ciaciacia}
\|\nabla \Psi_u\|_2^2 + \mu^2 \|\Psi_u\|_2^2 + q^2 \int u^2 |\Psi_u|^2 = q \int (\ell\nabla\theta-q{\bf A}_u)\cdot\Psi_u u^2 \geq 0.
\end{equation}
Finally
\begin{equation}
\ell  \int  \n \theta\cdot \Psi_u u^2 
= \int (\ell \n \theta-q\ba_u)\cdot\ba_u u^2. \label{cacca2a}
\end{equation}
\end{Lem}

\begin{proof}
	In the first part we proceed as in Lemma \ref{Lempropphi}. So we define the map $\mathcal{T}: \hat{H}^1_\sharp\times \mathcal{A}\to \mathcal{A}$ such that 
	\[
	\mathcal{T}(u,\ba)=q(\Delta-\mu^2)^{-1}[(\ell \n \theta-q\ba)u^2] + \ba.
	\]
	Simple calculations show that
	\[
	\partial_u \mathcal{T}(u,\ba)=v\in \hat{H}^1_\sharp \longmapsto 2q (\Delta - \mu^2)^{-1}[(\ell \n \theta-q\ba)uv]
	\]
	and
	\[
	\partial_\ba \mathcal{T}(u,\ba)={\bf V}\in\mathcal{A}\longmapsto -q^2(\Delta - \mu^2)^{-1}[{\bf V}u^2] + {\bf V}.
	\]
	We have that $\partial_\ba \mathcal{T}(u,\ba)$ is invertible and that
	\[
	(\partial_\ba \mathcal{T}(u,\ba))^{-1}=(\Delta - \mu^2 - q^2u^2)^{-1}\circ(\Delta-\mu^2).
	\]
	Thus, by Implicit Function Theorem, $\mathscr{A}$ is of class $\cC^1$ and, for all $v\in \hat{H}^1_\sharp$,
	\[
	\mathscr{A}'(u)[v]=-2q (\Delta - \mu^2 - q^2u^2)^{-1} [(\ell \n \theta-q\ba_u)uv];
	\]
	so
\[
\Psi_u=-q (\Delta - \mu^2 - q^2u^2)^{-1} [(\ell \n \theta-q\ba_u)u^2],
\]
namely $\Psi_u$ satisfies \eqref{caccaa}. 
Moreover, since $\ba_u$ solves the third equation in \eqref{P}, by \eqref{caccaa} we have
\[
q\int (\ell \n \theta-q\ba_u)\Psi_u u^2
=\int (-\Delta \ba_u + \mu^2 \ba_u)\Psi_u 
=\int(-\Delta\Psi_u + \mu^2 \Psi_u)\ba_u 
= q\int (\ell \n \theta-q\ba_u-q\Psi_u)\ba_u u^2
\]
and so \eqref{cacca2a}.
\end{proof}

Hence we can also consider a second reduced functional
\begin{equation}\label{J}
\begin{split}
\mathcal{J}(u)
&=
J(u,{\bf A}_u)\\
&=
\frac{1}{2} \|\nabla u\|_2^2
+\frac{m^2-\omega^2}{2}\|u\|_2^2
+ \frac{\ell^2}{2} \int \frac{u^2}{r^2}
+\frac{q\omega}{2}\int \phi_u u^2
-\frac{\ell q}{2}\int \nabla\theta\cdot{\bf A}_u u^2
- \int F(x,u)
\end{split}
\end{equation}
defined on $\hat{H}^1_\sharp$, which is of class $\cC^1$ by Lemmas \ref{Lempropphi} and \ref{le:A'}.


\subsection{The nonlinearity}
We conclude this section, showing some useful properties on the nonlinearity $f$. First of all we observe that if $f$ satisfies (\ref{SW1}) and (\ref{SW2}), then
\begin{equation}
\label{eq:estimF0}
\forall\varepsilon>0 \ \ \exists C_\eps >0 \hbox{ such that for a.e. }x\in\mathbb{R}^3 \hbox{ and } \forall s\in\mathbb{R} : |f(x,s)| \leq \eps |s| + C_\eps |s|^{p-1}
\end{equation}
and so, for all $u\in  H^1(\mathbb{R}^3)$,
\begin{equation}
\label{bdd1}
\int |f(x,u) u| \leq \eps \|u\|_2^2 + C_\eps \|u\|_p^p
\end{equation}
and
\begin{equation}
\label{eq:estFn}
\forall\varepsilon>0 \ \ \exists C_\eps >0 \hbox{ such that for a.e. }x\in\mathbb{R}^3 \hbox{ and } \forall u\in  H^1(\mathbb{R}^3) : 
\int F(x,u) \leq \eps \|u\|_2^2 + C_\eps \|u\|_p^p.
\end{equation}

Moreover we have
\begin{Lem}\label{Le29}
If $f$ satisfies (\ref{SW1})--(\ref{SW4}), for a.e. $x\in\mathbb{R}^3$ we have
\begin{equation}\label{eq:estimF1}
sf(x,s)\geq \sigma F(x,s)\geq 0
\hbox{ for all }
s\in\mathbb{R}
\end{equation}
and
\begin{equation}
\label{eq:estimF2}
F(x,ts)\geq F(x,s) t^\sigma
\hbox{ for all } s\in\mathbb{R},\ t\geq 1.
\end{equation}
If $f$ satisfies (\ref{SW1})--(\ref{SW3}) and (\ref{SW5}), for a.e. $x\in\mathbb{R}^3$ we have
\begin{equation}\label{eq:estimF1n}
sf(x,s)\geq 4 F(x,s)\geq 0
\hbox{ for all }
s\in\mathbb{R}
\end{equation}
and
\begin{equation}
\label{eq:estimF2n}
F(x,ts)\geq F(x,s) t^4
\hbox{ for all } s\in\mathbb{R},\ t\geq 1.
\end{equation}
\end{Lem}

\begin{proof}
It is easy to show that (\ref{SW4}) implies that $s\mapsto f (x,s)/|s|^{\sigma-1}$ is nondecreasing on $(-\infty, 0)$ and $(0,+\infty)$ and so \eqref{eq:estimF1} holds.
Moreover, since the function $t>0\mapsto F(x,ts)/t^\sigma $ is nondecreasing by \eqref{eq:estimF1}, we get \eqref{eq:estimF2}. Finally \eqref{eq:estimF1n} and \eqref{eq:estimF2n} can be obtained in the same way.
\end{proof}

\begin{Rem}
As in \cite[Lemma 9]{BF2009NA}, up to take $f(s)=0$ for $s<0$, we can show, using  \eqref{bohhh}, that the each solution $(u,\phi_u,{\bf A})$ of \eqref{P} has $u\geq 0$ a.e. in $\mathbb{R}^3$.
\end{Rem}

\section{Minimizing on the Nehari manifold}\label{section:Th1}
In this section we assume that (\ref{SW1})--(\ref{SW4}) hold and we look for minimizers of the reduced functional $\mathcal{J}$, defined in \eqref{J}, on its Nehari manifold
	\[
	\mathcal{N}:=
	\left\{u\in \hat{H}^1_\sharp\setminus\{0\} : \mathcal{J}'(u)[u]=0	\right\}.
	\]

First we prove
\begin{Lem}\label{lem:3_1}
For any $u\in \hat{H}^1_\sharp\setminus\{0\}$ there exists $\bar t>0$ such that  $\bar t u\in \mathcal{N}$ and so
$\mathcal{N}$ is nonempty.
\end{Lem}

\begin{proof}
	Let $u\neq 0$ and consider
	\begin{align*}
	j(t):
	&=
	\mathcal{J}(tu)\\
	&=
	\frac{t^2}{2} \|\nabla u\|_2^2
	+\frac{m^2-\omega^2}{2} t^2 \|u\|_2^2
	+\frac{\ell^2}{2} t^2 \int\frac{u^2}{r^2}
	+\frac{q\omega}{2}t^2\int \phi_{tu} u^2
	- \frac{\ell q}{2} t^2\int \nabla\theta\cdot{\bf A}_{tu} u^2
	\\
	&\quad
	- \int F(x,tu).
	\end{align*}
	We have that $j(0)= 0$ and, using also \eqref{Auineq} and \eqref{eq:estimF2}, we have
	\[
	j(t)\leq
	\frac{t^2}{2} \|\nabla u\|_2^2
	+\frac{m^2-\omega^2}{2} t^2 \|u\|_2^2
	+\frac{\ell^2}{2} t^2 \int\frac{u^2}{r^2}
	+\frac{q\omega}{2}t^2\int \phi_{tu} u^2
	- t^\sigma \int F(x,u) \to -\infty
	\]
	as $t\to +\infty$, being $\sigma>2$.
	Moreover, by \eqref{bohhh}, \eqref{comevuoi} and \eqref{eq:estFn},
	\[
	j(t)
	\geq
	\frac{t^2}{2} \|\nabla u\|_2^2
	+\frac{m^2-\omega^2}{2} t^2 \|u\|_2^2
	- \eps t^2 \|u\|_2^2 - C_\eps t^p \|u\|_p^p
	\geq
	C(t^2-t^p).
	\]
	Hence we get that $j$ admits a maximum point $\bar t>0$ and
	$0=j'(\bar t)= \mathcal{J}'(\bar tu)[u]$,
	so that $\bar tu\in\mathcal{N}$.
\end{proof}

Moreover we have
\begin{Lem}
	\label{le:disc}
There exists $C>0$ such that for every $u\in\mathcal{N}$, $\|u\|_p\geq C$.
\end{Lem}
\begin{proof}
	Let $u\in\mathcal{N}$, then, using \eqref{bohhh}, we have
	\begin{align*}
	\int f(x,u)u 
	&=
	\|\nabla u\|_2^2
	+m^2\| u\|_2^2  
	+ \int |\ell\nabla\theta - q {\bf A}_u|^2 u^2 
	- \int (\omega-q\phi_u)^2 u^2 \\
	&\geq
	\|\nabla u\|_2^2
	+(m^2-\omega^2)\| u\|_2^2  
	+q \int (2\omega-q\phi_u)\phi_u u^2 \\
	&\geq
	\|\nabla u\|_2^2
	+(m^2-\omega^2)\| u\|_2^2.
	\end{align*}
	Hence, by \eqref{bdd1} we have that
	\begin{equation*}
	\|\nabla u\|_2^2 + (m^2-\omega^2-\varepsilon) \|u\|_2^2 \leq C_\eps \|u\|_p^p 
	\end{equation*}
	 and using Sobolev inequality we conclude.
\end{proof}

We have also
\begin{Lem}\label{pr:nc}
	$\mathcal{N}$ is a natural constraint.
\end{Lem}


\begin{proof}
	Let $u\in \mathcal{N}$ be a critical point of $\mathcal{J}|_{\mathcal{N}}$, then there exists $\lambda\in \R$ such that $$\mathcal{J}'(u)=
	\lambda \partial_u (\mathcal{J}'(u)[u]).$$ Our aim is to show that $\lambda=0$. Since 
	\[
	0
	=
	\mathcal{J}'(u)[u]
	=
	\lambda \partial_u (\mathcal{J}'(u)[u]) [u],
	\]
we conclude if we prove that $\partial_u (\mathcal{J}'(u)[u]) [u]\neq 0$. We have that
	\begin{align*}
	\partial_u (\mathcal{J}'(u)[u]) [u]
	&=
	2 \|\nabla u\|_2^2
	+2 \int \left[m^2-(\omega-q\phi_u)^2 \right]u^2
	+ 2\int |\ell\nabla\theta - q {\bf A}_u|^2 u^2
	\\
	&\quad
	-4q\int (\ell\nabla\theta-q{\bf A}_u)\cdot\Psi_u u^2\
	+ 4q \int (\omega-q\phi_u)\psi_u u^2
	- \int f(x,u)u
	\\
	&\quad
	- \int \de_u f(x,u)u^2
	\\
	&=
	(2- \sigma) \|\nabla u\|_2^2
	+(2- \sigma)m^2\| u\|_2^2  
	+ (2- \sigma)\int |\ell\nabla\theta - q {\bf A}_u|^2 u^2 
	\\
	&\quad
	-4q\int (\ell\nabla\theta-q{\bf A}_u)\cdot\Psi_u u^2
	+ 4q \int (\omega-q\phi_u)\psi_u u^2 
	- (2- \sigma) \int (\omega-q\phi_u)^2 u^2   
	\\
	&\quad
	+ (\sigma-1) \int f(x,u)u 
	- \int \de_u f(x,u)u^2
	\\
	\hbox{(by \eqref{SW4} and \eqref{ciaciacia})}
	&\leq  
	(2- \sigma) \|\nabla u\|_2^2 
	+(2- \sigma)m^2\| u\|_2^2
	+ 4q \int (\omega-q\phi_u)\psi_u u^2  \\
	&\quad
	- (2- \sigma) \int (\omega-q\phi_u)^2 u^2 \\
	\hbox{(by \eqref{cacca2})}
	&=
	(2- \sigma) \|\nabla u\|_2^2 \\
	&\quad
	+\int \left[
	(2- \sigma)(m^2-\omega^2)
	+(2- \sigma) q\omega\phi_u
	+(6- \sigma) q \omega\psi_u
	- 4q^2\phi_u\psi_u
	\right]u^2
	\\
	\hbox{(by \eqref{cc3})}&\leq
	-(\sigma-2) \|\nabla u\|_2^2 
	-\int \left[
	(\sigma-2)(m^2-\omega^2)
	+2(\sigma-4) q \omega\psi_u
	+ 4q^2\psi_u^2
	\right]u^2.
	\end{align*}
	As in \cite{Wang}, simple calculations show that
	\[
	\left[
	(\sigma-2)(m^2-\omega^2)
	+2(\sigma-4) q \omega\psi_u
	+ 4q^2\psi_u^2
	\right]
	\geq
	\begin{cases}
	(\sigma-2)(m^2-\omega^2) & \hbox{for }\sigma\geq 4\\
	\displaystyle (\sigma-2)m^2 - \frac{\sigma^2 - 4 \sigma +8}{4}\omega^2 & \hbox{for }2<\sigma< 4
	\end{cases}
	\]
	a.e. in $\mathbb{R}^3$ and then, by \eqref{condomegamsigma} and Lemma \ref{le:disc}, we get $\partial_u (\mathcal{J}'(u)[u]) [u] <-C(\|\nabla u\|_2^2 + \|u\|_2^2)<-C<0$.
\end{proof}

\begin{Lem}\label{le:bbelow}
	The functional $\mathcal{J}$ is bounded from below on $\mathcal{N}$.
\end{Lem}
\begin{proof}
	Let $u\in\mathcal{N}$. By \eqref{eq:estimF1}, \eqref{Aueq} and \eqref{comevuoi} we have
	\begin{align*}
	\mathcal{J}(u)
	&\geq
	\frac{1}{2} \|\nabla u\|_2^2
	+\frac{m^2-\omega^2}{2}\|u\|_2^2
	+ \frac{\ell^2}{2} \int \frac{u^2}{r^2}
	+\frac{q\omega}{2}\int \phi_u u^2
	-\frac{\ell q}{2}\int \nabla\theta\cdot{\bf A}_u u^2
	- \frac{1}{\sigma} \int f(x,u) u 
	\\
	&=
	\frac{\sigma-2}{2\sigma} \|\nabla u\|_2^2
	+ \frac{\sigma-2}{2\sigma} (m^2-\omega^2) \|u\|_2^2
	+\frac{\sigma-4}{2\sigma}q\omega \int \phi_u u^2
	+ \frac{q^2}{\sigma} \int \phi_u^2 u^2
	\\
	&\quad
	+ \frac{\sigma-2}{2\sigma} \left[\ell^2  \int \frac{u^2}{r^2}
	-\ell q\int \nabla\theta\cdot{\bf A}_u u^2\right]
	+\frac{q}{\sigma} \int (\ell\nabla\theta-q{\bf A}_u)\cdot{\bf A}_u u^2
	\\
	&\geq
	\frac{\sigma-2}{2\sigma} \|\nabla u\|_2^2
	+\frac{1}{2\sigma}\int \left[
	(m^2-\omega^2)(\sigma-2)
	+(\sigma-4)q\omega \phi_u
	+ 2q^2 \phi_u^2 
	\right]u^2.
	\end{align*}
	Simple calculations show that
	\[
	\left[
	(m^2-\omega^2)(\sigma-2)
	+(\sigma-4)q\omega \phi_u
	+ 2q^2 \phi_u^2 
	\right]
	\ge
	\begin{cases}
	(\sigma-2)(m^2-\omega^2)
	& \hbox{for }\sigma\geq 4\\
	\displaystyle(\sigma-2)m^2 - \frac{\sigma^2}{8}\omega^2
	& \hbox{for }2<\sigma< 4
	\end{cases}
	\]
	a.e. in $\mathbb{R}^3$ and then, 
	by \eqref{condomegamsigma} and Lemma \ref{le:disc}, we conclude.
\end{proof}
Now we can complete the proof.
\begin{proof}[Proof of Theorem \ref{ThMain1} concluded]
	Let $\{u_n\}\subset\mathcal{N}$ be a minimizing sequence, i.e.
	\[
	\lim_n \mathcal{J} (u_n) = \inf_{u\in\mathcal{N}} \mathcal{J}(u).
	\]
	As a first step, we want to prove that $\{u_n\}$ is bounded in $\hat H^1$.
	\\
	Arguing as in the proof of Lemma \ref{le:bbelow} we get that $\{u_n\}$ is bounded in $H^1(\mathbb{R}^3)$.
	Hence, we need to prove the boundedness of $\{u_n\}$ in $\hat{H}^1$. If we set ${\bf A}_n={\bf A}_{u_n}$, by \eqref{eq:estFn} we have
	\begin{equation}
	\label{estimcalJ3}
	\mathcal{J}(u_n)
	\geq
	\frac{1}{2} \|\nabla u_n\|_2^2
	+\frac{m^2-\omega^2-\varepsilon}{2}\|u_n\|_2^2
	+\frac{\ell^2}{2}\int \frac{u_n^2}{r^2}
	-\frac{q\ell}{2} \int \nabla\theta \cdot{\bf A}_n u_n^2
	- C_\eps \|u_n\|_p^p
	\end{equation}
	and so, if $\eps$ small enough, since $\{u_n\}$ is bounded in $H^1(\mathbb{R}^3)$ and, by \eqref{comevuoi},
	we get that $\{{\bf A}_n\}$ is bounded in $(H_\mu)^3$ and so in $(L^6(\R^3))^3$. Moreover
	\begin{equation}
	\label{eq:mixterm3}
	\begin{split}
	\int q|{\bf A}_n|\frac{\ell}{r} u_n^2
	& \leq
	8 q^2\int |{\bf A}_n|^2 u_n^2 
	+ \frac{\ell^2}{8} \int \frac{u_n^2}{r^2}
	\\
	& \leq
	C \|{\bf A}_n\|_6^2 \|u_n\|_3^2
	+ \frac{\ell^2}{8} \int \frac{u_n^2}{r^2}
	\\
	& \leq
	C
	+ \frac{\ell^2}{8} \int \frac{u_n^2}{r^2}
	\end{split}
	\end{equation}
	and so, combining \eqref{eq:mixterm3} with \eqref{estimcalJ3}, we get the boundedness of $\{u_n\}$ in $\hat{H}^1$, as desired.\\
	By Lemma \ref{le:disc} and by the classical Lions Lemma \cite[Lemma I.1]{Lions842}, there are $\delta>0$ and a sequence $\{z_n\}\subset \Z^3$,
	such that
	\begin{equation}\label{formulacentrata}
\int_{B_1(z_n)}|u_n(x)|^2 \geq \delta
\end{equation}
	for any $n$.  Let us consider the cylindrical group action $G=O(2)\times \operatorname{Id}\subset O(3)$ on $\R^3$. Observe that in the family $\{B_1(gz_n)\}_{g\in G}$ we find an increasing number of disjoint balls when $r_n=|z_n\cdot(e_1+e_2)|\to+\infty$, being $\{e_i\}_{i=1,2,3}$ the canonical basis of $\R^3$. 
	Since
$\{u_n\}$ is bounded in $L^2(\RT)$ and the functions $u_n$ are cylindrical symmetric, then, by \eqref{formulacentrata} $r_n$ must be bounded and so, for sufficiently large $R>0$, 
	\begin{equation}\label{eqsec4:vanish3}
	\int_{B_R(z_n^3  e_3)}|u_n(x)|^2\geq \delta
	\end{equation}
	where $z_n^3$ is the third component of $z_n$.
	Due also to \eqref{SW1}, the functional $\mathcal{J}$ is invariant with respect to $\Z$-translations in the $x_3$-axis and so the sequence $\{u_n(\cdot+z_n^3 e_3)\}$, that we denote again $\{u_n\}$, is still a minimizing sequence.
	Thus, in view of \eqref{eqsec4:vanish3}, we get that, up to a subsequence, $u_n\rightharpoonup u_0\neq 0$ in $\hat H^1_\sharp$.
\\
	As a second step, we show that $(u_0, \phi_{u_0}, {\bf A}_{u_0})$ is a solution of \eqref{P}. 
	\\
	Since $\mathcal{N}$ is a natural constraint, by \cite[Theorem 8.5]{Willem} we have that $\{u_n\}$ is a Palais-Smale sequence for $\mathcal{J}$ ((PS) sequence, for short) and $\mathcal{J}' (u_n) [v] \to 0$ for any $v\in \hat H_\sharp^1$. 
	Then, 
	\begin{align}
	&\partial_u I(u_n, \phi_n, {\bf A}_n) [v] \to 0,
	\hbox{ for any } v\in (C_0^\infty(\RT\setminus \Sigma))_\sharp,
	\label{139a}
	\\
	&\partial_\phi I(u_n, \phi_n, {\bf A}_n) [w] = 0,
	\hbox{ for any } w\in C_0^\infty(\RT),
	\label{140a}
	\\
	&\partial_{\bf A} I(u_n, \phi_n, {\bf A}_n)  [{\bf V}] = 0
	\hbox{ for any } {\bf V}\in \mathcal{A}_0^\infty,
	\label{141a}
	\end{align}
	where $\phi_n=\phi_{u_n}$.
	Since $\{u_n\}$ is bounded in $\hat H^1$, by Lemma \ref{le:phipsi}, Lemma \ref{le:A} and \eqref{eq:mixterm3},
	also $\{\phi_n\}$ and $\{{\bf A}_n\}$ are bounded respectively in $H_\mu$ and $(H_\mu)^3$. Thus	there exist $\phi_0\in H_\mu$ and ${\bf A}_0\in (H_\mu)^3$ such that, up to a subsequence,
	\begin{align*}
	&\phi_n \rightharpoonup \phi_0
	\hbox{ weakly in } H_\mu,\\
	&{\bf A}_n\rightharpoonup{\bf A}_0
	\hbox{ weakly in } (H_\mu)^3.
	\end{align*}
Arguing as in \cite[Lemma 2.7]{AP}, the weak convergence of $\{u_n\}$, \eqref{140a} and \eqref{141a} imply
\begin{align*}
	\partial_\phi I(u_0, \phi_0, {\bf A}_0) [w] &= 0 \hbox{ for any } w\in C_0^\infty(\RT),
\\
	\partial_{\bf A} I(u_0, \phi_0, {\bf A}_0)  [{\bf V}] &= 0 \hbox{ for any } {\bf V}\in \mathcal{A}_0^\infty,
\end{align*}
	and so, by the uniqueness results in Section \ref{se:vs}, $\phi_0=\phi_{u_0}$ and ${\bf A}_0={\bf A}_{u_0}$.
Moreover, by \eqref{139a}, 
we get
\[
	\partial_u I(u_0, \phi_0, {\bf A}_0) [v] = 0
	\hbox{ for any } v\in (C_0^\infty(\RT\setminus \Sigma))_\sharp.
	\]
	Hence, 
	\begin{align*}
	&\partial_u J(u_0, {\bf A}_{u_0}) [v]=\partial_u I(u_0, \phi_0, {\bf A}_0) [v] = 0
	\hbox{ for any } v\in (C_0^\infty(\RT\setminus \Sigma))_\sharp,
	\\
	&\partial_{\bf A} J(u_0, {\bf A}_{u_0}) [{\bf V}]=\partial_{\bf A} I(u_0, \phi_0, {\bf A}_0)  [{\bf V}] = 0
	\hbox{ for any } {\bf V}\in \mathcal{A}_0^\infty,
	\end{align*}
	and so, by Lemma \ref{eq:lemNaturalConstr}, we have that $(u_0,\phi_0, {\bf A}_0)$ is a solution for \eqref{P} and so $u_0\in\mathcal{N}$.\\
	Finally, to prove that $u_0$ is a ground state, we observe that
	\begin{align*}
	\mathcal{J}|_\mathcal{N}(u)
	&=
	\frac{\sigma-2}{2\sigma} \|\nabla u\|_2^2
	+ \frac{\sigma-2}{2\sigma} (m^2-\omega^2) \|u\|_2^2
	+\frac{\sigma-4}{2\sigma}q\omega \int \phi_u u^2
	+ \frac{q^2}{\sigma} \int \phi_u^2 u^2
	\\
	&\quad
	+ \frac{\sigma-2}{2\sigma}
	\left[
	\|\nabla\times{\bf A}_u\|_2^2
	+ \mu^2 \|{\bf A}_u\|_2^2 
	+\int |\ell\nabla\theta- q{\bf A}_u|^2 u^2\right]
	+\frac{1}{\sigma} \left[ \|\nabla\times{\bf A}_u\|_2^2 + \mu^2 \|{\bf A}_u\|_2^2 \right]
	\\
	&\quad
	+ \frac{1}{\sigma} \int [f(x,u) u - \sigma F(x,u)]
	\end{align*}
	and then, arguing as at the end of the proof of Lemma \ref{le:bbelow}, by Fatou Lemma and the weak convergences, we can conclude that
	\[
	\inf_\mathcal{N} \mathcal{J}(u) \leq \mathcal{J}(u_0) \leq \liminf_n \mathcal{J}(u_n) = \inf_\mathcal{N} \mathcal{J}(u).
	\]
\end{proof}

\section{A Mountain Pass approach}\label{se:MP}
In this section we assume \eqref{SW1}--\eqref{SW3} and \eqref{SW5}. Note that, arguing as in Lemma \ref{lem:3_1}, we can show that $\cN$ is nonempty. However, it may be not of class $\cC^1$ and so the minimization technique from Section \ref{section:Th1} fails. Moreover, due to the nonlocal terms $\phi_u$ and $\ba_u$ it is not clear if for any $u\in \hat{H}^1_{\sharp}$  the map $\J$ attains its maximum on $\R^+u$ at an unique point $tu\in \cN$ with $t\geq 0$. This is a crucial property to prove that the Nehari manifold is a topological manifold homeomorphic to the unit sphere, where the minimization techniques can be performed in the spirit of \cite{SzulkinWeth}.

In order to overcome the above difficulties we introduce a larger constraint $\cM$ 
and we intend to apply the following variant of the Mountain Pass Theorem.
\begin{Lem}\label{lemMPT}
Let $X$, $Y$ be Banach spaces, $J:X\times Y\to\R$ is of class $\cC^1$ and
\begin{equation*}
\cM:=\{(u,v)\in (X\setminus\{0\})\times (Y\setminus\{0\})|\; \partial_u J(u,v)[u]=0,\;\partial_v J(u,v)[v]=0\}\neq\emptyset.
\end{equation*}
Let us assume that $J$ satisfies the following
\begin{enumerate}[label=(J\arabic*),ref=J\arabic*]
\item \label{J1}there is $\rho>0$ such that 
$$\inf_{\|(u,v)\|_{X\times Y}=\rho}J>0=J(0,0);$$
\item \label{J2}for any $(u,v)\in \cM$ there is $T>0$ such that $\|(Tu,v)\|_{X\times Y}>\rho$ and $J(Tu,v)<0$;
\item \label{J3}if $(u,v)\in\cM$ then
$$J(u,v)\geq \max_{(t,s)\in[0,+\infty)\times[0,1]}\{J(tu,v),J(0,sv)\}.$$
\end{enumerate}
Then, if
$$\Gamma:=\left\{\gamma\in \cC([0,1],X\times Y) : \gamma(0)=(0,0), J(\gamma(1)) < 0, \|\gamma(1)\|_{X\times Y}>\rho \right\},$$
we have that
\begin{equation}\label{MPTlevel}
0<c:= \inf_{\gamma\in\Gamma} \max_{t\in [0,1]} J(\gamma(t))\leq \inf_{\mathcal{M}} J
\end{equation}
and there is a (PS) sequence $\{(u_n,v_n)\}$ at level $c$.
\end{Lem}
\begin{proof}
Since $J$ satisfies the assumptions of the classical Mountain Pass Theorem, then $c>0$ and there is a (PS) sequence $(u_n,v_n)$ at level $c$.
Let $(u,v)\in\cM$ and take $T>0$ such that $\|(Tu,v)\|_{X\times Y}>\rho$ and $J(Tu,v)<0$.  Let us consider a path
\[
\gamma(t)
=
\begin{cases}
(0,2tv) 
& 0\leq t\leq 1/2, \\
(T(2t-1)u,v)
& 1/2  \leq t\leq 1.\\
\end{cases}
\]
Then $\gamma\in\Gamma$ and by (\ref{J3})
$$c\leq \max_{t\in [0,1]}J(\gamma(t))\leq J(u,v).$$
Therefore  $c\leq \inf_{\mathcal{M}} J$.
\end{proof}

We take $X:=\hat{H}^1_\sharp$,  $Y:=\mathcal{A}$ and below we show that in our case, the functional $J$ defined in \eqref{defJ} satisfies (\ref{J1})--(\ref{J3}) in Lemma \ref{lemMPT} for
\[
\mathcal{M}=\{(u,\ba)\in (X\setminus\{0\})\times (Y\setminus\{0\})|\; \partial_u J(u,\ba)[u]=0,\;\partial_\ba J(u,\ba)[\ba]=0\}
\]
where
\begin{align*}
	\partial_u J(u,{\bf A})[u]
	&=
	\|\nabla u\|_2^2+m^2\| u\|_2^2  + \int |\ell\nabla\theta 
	- q {\bf A}|^2 u^2 - \int (\omega-q\phi_u)^2 u^2 - \int f(x,u)u,\\
	\partial_{\bf A} J(u,{\bf A})[{\bf A}] 
	&=
	\|\nabla\times{\bf A}\|_2^2 +\mu^2\|{\bf A}\|_2^2 +q^2\int |{\bf A}|^2 u^2
	-q\ell\int \nabla\theta\cdot{\bf A} u^2.
\end{align*}
Observe that if $u \in \cN$, then $(u,\ba_u)\in\cM$; hence, arguing as in Lemma \ref{lem:3_1}, $\cM$ is nonempty as well.

\begin{Lem}\label{lem:J1_J2}
The functional $J$ satisfies (\ref{J1}) and (\ref{J2}).
\end{Lem}
\begin{proof}
Observe that there exists $C>0$ such that for all $u\in \hat{H}^1_\sharp$ and $\ba\in \mathcal{A}$
\begin{equation*}
\begin{split}
\int |\ell\nabla\theta- q{\bf A}|^2 u^2 
&\geq
\int\left(\frac{\ell^2}{r^2} - 2\frac{q\ell}{r}|{\bf A}|+q^2|{\bf A}|^2\right) u^2
\\
&\geq
\int\left(\frac{\ell^2}{2r^2} - q^2|{\bf A}|^2\right) u^2
\\
&\geq
\frac{\ell^2}{2}\int\frac{u^2}{r^2}  - q^2\|{\bf A}\|_6^2 \|u\|_3^2
\\
&\geq
\frac{\ell^2}{2}\int\frac{u^2}{r^2}  - \frac{1}{2} q^4\|{\bf A}\|_6^4 - \frac{1}{2} \|u\|_3^4
\\
&\geq
\frac{\ell^2}{2} \int \frac{u^2}{r^2} - C [(\|\nabla\times {\bf A}\|_2^2+\mu^2\|{\bf A}\|_2^2)^2 + \|u\|^4].
\end{split}
\end{equation*}
Thus, by \eqref{eq:estFn}, 
\begin{align*}
J(u,{\bf A})
&\geq
\frac{1}{2} \|\nabla u\|_2^2
+\frac{m^2-\omega^2-\eps}{2}\|u\|_2^2
+ \frac{1}{2}\|\nabla\times{\bf A}\|_2^2
+\frac{\mu^2}{2}\|{\bf A}\|_2^2
+ \frac{\ell^2}{4} \int \frac{u^2}{r^2} - C (\|\nabla\times {\bf A}\|_2^2+\mu^2\|{\bf A}\|_2^2)^2
\\
&\quad - C \|u\|^4
-C_\eps \|u\|_p^p
\\
&\geq
C \| u \|^2(1- \|u\|^2 - \|u\|^{p-2}) +C (\|\nabla \times {\bf A}\|_2^2+\mu^2\|{\bf A}\|_2^2)\Big[1-(\|\nabla\times {\bf A}\|_2^2+\mu^2\|{\bf A}\|_2^2)\Big],
\end{align*}
from which we have (\ref{J1}).\\
Finally take $(u,{\bf A})\in V$ with $u\neq 0$ and note that by \eqref{bohhh}, for any $t>0$,
\begin{align*}
J(tu,{\bf A})
&=
\frac{t^2}{2} \|\nabla u\|_2^2
+\frac{m^2-\omega^2}{2}t^2\|u\|_2^2
+\frac{q\omega}{2}t^2\int \phi_{tu} u^2
+ \frac{\ell^2}{2} t^2\int \frac{u^2}{r^2}
+\frac{1}{2}\|\nabla\times{\bf A}\|_2^2
+\frac{\mu^2}{2}\|{\bf A}\|_2^2\\
&\quad+ \frac{t^2}{2} \int |\ell\nabla\theta- q{\bf A}|^2 u^2
- \int F(x,tu)\\
&\leq 
 t^2\left(\frac{1}{2} \|\nabla u\|_2^2
+\frac{m^2}{2}\|u\|_2^2
+ \frac{\ell^2}{2} \int \frac{u^2}{r^2}
+\frac{1}{2} \int |\ell\nabla\theta- q{\bf A}|^2 u^2
- \int \frac{F(x,tu)}{t^2}\right)\\
&\quad +\frac{1}{2}\|\nabla\times{\bf A}\|_2^2
+\frac{\mu^2}{2}\|{\bf A}\|_2^2.
\end{align*}
In view of 
\eqref{eq:estimF2n} we have that
\[
\int \frac{F(x,tu)}{t^2}\geq t^2\int F(x,u) 
\] 
and, by (\ref{SW3}), we can conclude that for $t$ large enough, $J(tu,{\bf A})<0$.
\end{proof}

\begin{Lem}
The functional $J$ satisfies (\ref{J3}). 
\end{Lem}
\begin{proof}
Let us assume $(u,{\bf A})\in\cM$ and let
$$j(t)
=
J(tu,{\bf A}).$$
Observe that, if $t>0$
\begin{align*}
j'(t)
&=
\partial_u J(tu,{\bf A})[u]\\
&=
t \|\nabla u\|_2^2
+t(m^2-\omega^2) \|u\|_2^2
+ t\int |\ell\nabla\theta- q{\bf A}|^2 u^2
+2q\omega t \int \phi_{tu} u^2
-q^2 t \int \phi_{tu}^2 u^2
- \int f(x,tu)u\\
&=
t^3
\underbrace{\left[
\frac{1}{t^2}
\left(
\|\nabla u\|_2^2
+(m^2-\omega^2) \|u\|_2^2
+ \int |\ell\nabla\theta- q{\bf A}|^2 u^2
\right)
+2q\omega \int \frac{\phi_{tu}}{t^2} u^2
-q^2 \int \frac{\phi_{tu}^2}{t^2} u^2
- \int \frac{f(x,tu)}{t^3}u
\right]}_{\bar{j}(t)}
\end{align*}
The function $\bar{j}$ is nonincreasing by (\ref{SW5}) and since by \eqref{cacca2} and \eqref{cc3}
we have
\begin{align*}
\frac{d}{dt}\left(
2\omega \int \frac{\phi_{tu}}{t^2} u^2
-q \int \frac{\phi_{tu}^2}{t^2} u^2\right)
&=
\int \frac{4\omega\psi_{tu}-4q\phi_{tu}\psi_{tu} -4\omega\phi_{tu}+2q\phi_{tu}^2}{t^3}u^2\\
&=
-2q\int \frac{\phi_{tu}^2+2\phi_{tu}\psi_{tu}}{t^3}u^2\leq 0.
\end{align*} 
Moreover, since $(u,\ba)\in \cM$, $\bar{j}(1)=0$ and so we obtain that
\begin{equation}\label{scemochilegge}
J(u,{\bf A})\geq J(tu,{\bf A}) \hbox{ for any } t\geq 0.
\end{equation}
Finally, 
since $(u,{\bf A})\in\cM$ and by \eqref{scemochilegge}, for any $s\in [0,1]$, 
$$J(0,s{\bf A})\leq J(0,{\bf A})\le J(u,\ba).$$
\end{proof}

Now we prove the following results that will be useful to get Theorem \ref{ThMain2}.

\begin{Lem}\label{lemPSbounded}
Every (PS) sequence for the functional $J$ is bounded.
\end{Lem}
\begin{proof}
Let $\{(u_n,{\bf A}_n)\}$ be a (PS) sequence for the functional $J$, i.e.
\[
J(u_n,{\bf A}_n)\to \beta 
\]
and
\begin{equation*}
dJ(u_n,{\bf A}_n)\to 0 \hbox{ in } V'
\end{equation*}
as $n\to+\infty$.
Using \eqref{eq:estimF1n}, we have that
\begin{align*}
4 J(u_n,{\bf A}_n)-\partial_u J(u_n,{\bf A}_n)[u_n]
&=
\|\nabla u_n\|_2^2 
+(m^2-\omega^2)\|u_n\|_2^2
+ 2\|\nabla\times{\bf A}_n\|_2^2
+ 2 \mu^2 \|{\bf A}_n\|_2^2\\
& \quad
+q^2 \int \phi_n^2 u_n^2
+ \int |\ell \nabla\theta-q{\bf A}_n|^2 u_n^2
+\int [f(x,u_n)u_n - 4 F(x,u_n)]\\
&\geq
\|\nabla u_n\|_2^2 
+(m^2-\omega^2)\|u_n\|_2^2
+ 2\|\nabla\times{\bf A}_n\|_2^2
+ 2 \mu^2 \|{\bf A}_n\|_2^2.
\end{align*}
On the other hand
\[
4 J(u_n,{\bf A}_n)-\partial_u J(u_n,{\bf A}_n)[u_n]
\leq
4\beta+1+o_n(1) \sqrt{\|\nabla u_n\|_2^2 +(m^2-\omega^2)\|u_n\|_2^2}.
\]
Then
\[
\|\nabla u_n\|_2^2 
+(m^2-\omega^2)\|u_n\|_2^2
+ 2\|\nabla\times{\bf A}_n\|_2^2
+ 2 \mu^2 \|{\bf A}_n\|_2^2
\leq
4\beta+1+o_n(1) \sqrt{\|\nabla u_n\|_2^2 +(m^2-\omega^2)\|u_n\|_2^2}
\]
and so $\{u_n\}$ is bounded in $H^1(\mathbb{R}^3)$ and $\{{\bf A}_n\}$ is bounded in $(H_\mu)^3$.
Then we can conclude arguing as in the proof of Theorem \ref{ThMain1}.
\end{proof}

\begin{Lem}\label{le:19}
If $\{(u_n,{\bf A}_n)\}$ is a (PS) sequence for $J$ at level $\beta>0$, then $\|u_n\|_p\geq C >0$.
\end{Lem}
\begin{proof}
Take any (PS) sequence $\{(u_n,{\bf A}_n)\}$ at level $\beta>0$.
Assume by contradiction that $u_n\to 0$ in $L^p(\mathbb{R}^3)$.
Then, the interpolation and Sobolev inequalities and the boundedness of the (PS) sequences imply that $u_n\to 0$ in $L^s(\mathbb{R}^3)$ for all $s\in[2,6)$.
	Moreover, since $\partial_{u} J(u_n,{\bf A}_n)[u_n]=o_n(1)$, $\{u_n\}$ and $\{{\bf A}_n\}$ are bounded and by H\"older and Sobolev inequalities, 
	\eqref{eq:mixterm3}, 
	Lemma \ref{le:phipsi} 
	and \eqref{eq:estimF0} we have
	\begin{align*}
	\partial_{u} J(u_n,{\bf A}_n)[u_n]
	+ q^2 \int \phi_n^2 u_n^2
	&\geq
	C\left(\|\nabla u_n\|_2^2+(m^2-\omega^2)\| u_n\|_2^2  + \ell^2 \int \frac{u_n^2}{r^2}
	- \|{\bf A}_n\|_6^2 \|u_n\|_3^2 - \|u_n\|_p^p\right)\\
	&=
	C\left(\|\nabla u_n\|_2^2+(m^2-\omega^2)\| u_n\|_2^2  + \ell^2 \int \frac{u_n^2}{r^2}\right)
	+o_n(1)
	\end{align*}
	and
	\[
	\partial_{u} J(u_n,{\bf A}_n)[u_n]
	+ q^2 \int \phi_n^2 u_n^2
	=o_n(1)
	+ q^2 \int \phi_n^2 u_n^2\\
	\leq
	o_n(1)
	+C\|u_n\|_2^2 = o_n(1)
	\]
	and so $u_n \to 0$ in $\hat{H}^1$.
	Moreover, since $\partial_{\bf A} J(u_n,{\bf A}_n)[{\bf A}_n]=o_n(1)$, arguing as before,
	\[
	\|\nabla\times{\bf A}_n\|_2^2 +\mu^2\|{\bf A}_n\|_2^2
	+ q^2\int |{\bf A}_n|^2 u_n^2 
	=o_n(1)
	+q\ell\int \nabla\theta\cdot{\bf A}_n u_n^2=o_n(1).
	\]
	Hence $	(u_n,{\bf A}_n)\to (0, {\bf 0}) $ in  $V$ and, by continuity, we get $J(u_n,{\bf A}_n)\to 0$ which contradicts the fact that $J(u_n,{\bf A}_n)\to \beta>0$.
\end{proof}

Thus we can conclude as follows.

\begin{proof}[Proof of Theorem \ref{ThMain2}]
In view of the previous lemmas given in this section, there exists $\{(u_n,{\bf A}_n)\}$, a bounded (PS) sequence for the functional $J$ at level $c>0$, with $c$ defined as in \eqref{MPTlevel}, such that $\|u_n\|_p\geq C >0$. 
Arguing as in the proof of Theorem \ref{ThMain1},
we get that, up to a subsequence 
$u_n\rightharpoonup u_0\neq 0$ in $\hat H^1_\sharp$ and ${\bf A}_n\rightharpoonup {\bf A}_0$ in $\mathcal{A}$ and
we show that $(u_0, \phi_{u_0}, {\bf A}_0)$ is a nontrivial critical point of $I$. 
Finally, we can conclude that $(u_0, {\bf A}_0)\in V$ is a nontrivial critical point of $J$ in $V$. In view of Lemma \ref{eq:lemNaturalConstr} and Proposition \ref{PropDitribSol}, $(u_0, \phi_{u_0}, {\bf A}_0)$ is a solution of \eqref{P} with $u_0\neq 0$,  $\phi_{u_0}\neq 0$ and  ${\bf A}_0\neq 0$. Observe that
\begin{align*}
c+o_n(1)
&=J(u_n,{\bf A}_n)\\
&=J(u_n,{\bf A}_n)-\frac14 \de_u J(u_n,{\bf A}_n)[u_n]
\\
&=
\frac14\|\nabla u_n\|_2^2 
+\frac14(m^2-\omega^2)\|u_n\|_2^2
+ \frac12\|\nabla\times{\bf A}_n\|_2^2
+ \frac12 \mu^2 \|{\bf A}_n\|_2^2
+\frac14 q^2 \int \phi_{u_n}^2 u_n^2
\\
& \qquad
+ \frac12 \int |\ell \nabla\theta-q{\bf A}_n|^2 u_n^2
+\frac 14 \int \left[f(x,u_n)u_n-  4F(x,u_n)\right],
\end{align*}
and hence  by Fatou Lemma, Lemma \ref{lemMPT} and the weak convergences, we get 
\[
c \geq J(u_0,{\bf A}_0)-\frac14 \de_u J(u_0,{\bf A}_0)[u_0] =J(u_0,{\bf A}_0)\geq \inf_{\mathcal{M}}J\geq c.
\]
\end{proof}

\section{The behaviour as $\mu\to 0$}\label{se:mu}

In this section we assume that $|\mu|\leq 1$, $F$ satisfies \eqref{SW1}-\eqref{SW3} and, additionally, \eqref{SW4} or \eqref{SW5},  and by 
the subscript or the superscript $\mu$ we denote the dependence on $\mu$ of functions, sets and functionals defined and used before.
For $\mu\neq 0$, the norms of  $V_{\mu}$'s are equivalent and so $V_{\mu}=V_1$. Moreover we have
$V_1\subset V_0$. 
\begin{proof}[Proof of Theorem \ref{ThMain3}]
Let $\mu\neq 0$. Note that
\[
{\bf K}_\mu({\bf A}^{\mu}_u)
\leq  {\bf K}_\mu({\bf A}^{1}_u)\leq {\bf K}_1({\bf A}^{1}_u),
\]
where ${\bf K}$ is defined in \eqref{KAPPA}.
Hence, by \eqref{bohhh} and the above inequality we have
\begin{align*}
\J_{\mu}(u) 
&=
\frac{1}{2} \|\nabla u\|_2^2
+\frac{m^2-\omega^2}{2}\|u\|_2^2
+\frac{q\omega}{2}\int \phi_u^{\mu} u^2
+ \frac{\ell^2}{2}\int\frac{u^2}{r^2}+ {\bf K}_\mu({\bf A}^{\mu}_u)
- \int F(x,u)\\
&\leq 
\frac{1}{2} \|\nabla u\|_2^2
+\frac{m^2}{2}\|u\|_2^2
+ \frac{\ell^2}{2}\int\frac{u^2}{r^2}+ {\bf K}_1({\bf A}^{1}_u)
- \int F(x,u)\\
&\leq
\frac{1}{2} \|\nabla u\|_2^2
+\frac{m^2}{2}\|u\|_2^2
+ \frac{\ell^2}{2}\int\frac{u^2}{r^2}
- \int F(x,u)
:=\tilde{\J}(u),
\end{align*}
being, by \eqref{Aueq} and \eqref{Auineq}, $
{\bf K}_1({\bf A}^{1}_{u})
\leq 0$.
\\
Let $u_{\mu} $ be a ground state of $\J_{\mu}$ and fix $u_0\neq 0$. By Lemma \ref{lem:3_1} we find $t_{\mu}>0$ such that $t_\mu u_0\in\cN_{\mu}$ and, using \eqref{eq:estimF2} if $f$ satisfies (\ref{SW4}), respectively \eqref{eq:estimF2n} if $f$ satisfies (\ref{SW5}), we have
$$0<\J_{\mu}(u_{\mu})\leq \J_{\mu}(t_\mu u_0)\leq 
\sup_{t\in [0,+\infty)}\tilde{\J}(t u_0)<+\infty.$$
Therefore $\{\J_{\mu}(u_{\mu})\}$ is bounded.
Moreover arguing as in the proof of Theorem \ref{ThMain1}, we infer that 
$\{u_{\mu}\}$ is bounded in $\hat{H}^1$ and $\{\ba_{u_\mu}^\mu\}$ in $(\D^{1,2}(\R^3))^3$.
%
%
%
\\
Arguing as in Lemma \ref{le:disc}, we infer that
\begin{equation*}\label{lem41:eq}
\inf_{\mu\neq 0} \|u_{\mu}\|_p >0
\end{equation*}
and so, as in the proof of Theorem \ref{ThMain1}, we find a family $\{y_{\mu}\}\subset\Z$ such that, denoting again $u_\mu$ the function $u_\mu(\cdot+y_\mu e_3)$, we have that
	\[
	u_{\mu}\rightharpoonup u_0\neq 0 \hbox{ in }\hat{H}^1_\sharp \hbox{ as }\mu\to 0.
	\]
	Since, by Lemma \ref{le:phipsi}, $\{\phi^\mu_{u_\mu}\}$ is bounded also in $\D^{1,2}(\R^3)$,
	\[
	\phi_{u_{\mu}}^{\mu}\rightharpoonup \phi_0 \hbox{ in } \D^{1,2}(\R^3)
	\hbox{ and }
	{\bf A}^\mu_{u_\mu}\rightharpoonup {\bf A}_0 \hbox{ in } (\D^{1,2}(\R^3))^3
	\]
	as $\mu\to 0$ for some $\phi_0\in (\D^{1,2}(\R^3))_\sharp$ and $\ba_0\in\mathcal{A}_0$.
Then, since for every $\mu\neq 0$
\begin{align*}
&\partial_u I_\mu(u_\mu, \phi_{u_{\mu}}^{\mu}, {\bf A}^\mu_{u_\mu}) [v] = 0,
\hbox{ for any } v\in (C_0^\infty(\RT\setminus \Sigma))_\sharp,
\\
&\partial_\phi I_\mu(u_\mu, \phi_{u_{\mu}}^{\mu}, {\bf A}^\mu_{u_\mu}) [w] = 0,
\hbox{ for any } w\in C_0^\infty(\RT),
\\
&\partial_{\bf A} I_\mu(u_\mu, \phi_{u_{\mu}}^{\mu}, {\bf A}^\mu_{u_\mu})  [{\bf V}] = 0
\hbox{ for any } {\bf V}\in \mathcal{A}_0^\infty,
\end{align*}
passing to the limit as $\mu\to0$ and the boundedness of $\{\phi^\mu_{u_\mu}\}$ and $\{\ba^\mu_{u_\mu}\}$
we get that $\phi_0=\phi_{u_0}^0$, $\ba_0=\ba_{u_0}^0$ and $(u_0, \phi_{u_0}^0, \ba_{u_0}^0)$ is a nontrivial solution of \eqref{P} for $\mu=0$.
\end{proof}

\begin{Rem}
We are not able to say if such {\em limit solution} $(u_0,\phi^0_{u_0},{\bf A}^0_{u_0})$ is a ground state for \eqref{P} with $\mu\neq 0$. One of the main difficulties is that, in general,  
it is not clear if, for a fixed $u\in\hat{H}^1_\sharp$, the map 
$\mu \mapsto \mathcal{J}_\mu (u)$ is increasing or not. In fact, if we write
\begin{equation}
\label{lastrem}
\mathcal{J}_\mu(u)
=\frac{1}{2} \|\nabla u\|_2^2
+\frac{m^2}{2}\|u\|_2^2
-K_\mu (\phi_u^{\mu})
+ \frac{\ell^2}{2}\int\frac{u^2}{r^2}
+ {\bf K}_\mu({\bf A}^{\mu}_u)
- \int F(x,u),
\end{equation}
we have that, if $\mu_1 \leq \mu_2$,
\[
{\bf K}_{\mu_1}({\bf A}^{\mu_1}_u)
\leq  {\bf K}_{\mu_1}({\bf A}^{\mu_2}_u)
\leq {\bf K}_{\mu_2}({\bf A}^{\mu_2}_u)
\]
and
\[
K_{\mu_1} (\phi_u^{\mu_1})
\leq K_{\mu_1} (\phi_u^{\mu_2})
\leq K_{\mu_2} (\phi_u^{\mu_2})
\]
but, in \eqref{lastrem}, they appear with the opposite sign.

\end{Rem}

\begin{Rem}
Observe that we can get a similar result as in Theorem \ref{ThMain3}, for any family $\{(u_{\mu},\phi_{\mu}, {\bf A}_{\mu})\}$ of solutions of \eqref{P} such that $\{I_\mu(u_{\mu},\phi_{\mu}, {\bf A}_{\mu})\}$ is bounded above.
\end{Rem}

\subsection*{Acknowledgement}
This work has been partially carried out during a stay of P.D. in Torun and of J.M. in Bari.
They would like to
express their deep gratitude to the Departments
for the support and warm hospitality. 
Moreover the authors would like to thank Professor Donato Fortunato for stimulating discussions.


\begin{thebibliography}{99}
	

\bibitem{APP}
A. Azzollini, L. Pisani, A. Pomponio, {\em Improved estimates and a limit case for the electrostatic Klein-Gordon-Maxwell system}, Proc. Roy. Soc. Edinburgh Sect. A {\bf 141} (2011), 449--463.

\bibitem{AP}
A. Azzollini, A. Pomponio, {\em Ground state solutions for the nonlinear Klein-Gordon-Maxwell equations}, Topol. Methods Nonlinear Anal. {\bf 35} (2010), 33--42.

\bibitem{BBS}
J. Bellazzini, C. Bonanno, G. Siciliano, {\em Magneto-static vortices in two dimensional abelian gauge theories}, Mediterr. J. Math. {\bf 6} (2009), 347--366. 

\bibitem{BF2002}
V. Benci, D. Fortunato, {\em Solitary waves of the nonlinear Klein-Gordon field equation coupled with the Maxwell equations}, Rev. Math. Phys. {\bf 14} (2002), 409--420.
	
\bibitem{BF2009NA}
V. Benci, D. Fortunato, {\em Three-dimensional vortices in Abelian gauge theories}, Nonlinear Anal. {\bf 70} (2009), 4402--4421.

\bibitem{BF2010CMP}
V. Benci, D. Fortunato, {\em Spinning Q-Balls for the Klein-Gordon-Maxwell Equations}, Commun. Math. Phys. {\bf 295} (3) (2010), 639--668.
%
%


\bibitem{BFbook}
V. Benci, D. Fortunato, {\em Variational methods in nonlinear field equations. Solitary waves, hylomorphic solitons and vortices}, Springer Monographs in Mathematics, Springer, Cham, 2014. 

\bibitem{BDP}
D. Bonheure, P. d'Avenia, A. Pomponio, {\em On the electrostatic Born-Infeld equation with extended charges}, to appear on Comm. Math. Phys., DOI: \href{http://dx.doi.org/10.1007/s00220-016-2586-y}{10.1007/s00220-016-2586-y}.

%

\bibitem{CGM} M. Clapp, M. Ghimenti, A.M. Micheletti, {\em Semiclassical states for a static supercritical Klein-Gordon-Maxwell-Proca system on a closed Riemannian manifold}, to appear on Commun. Contemp. Math. (DOI: 10.1142/S021919971550039X).

	
\bibitem{DM} 
T. D'Aprile, D. Mugnai, {\em Non-Existence Results for the Coupled 	Klein-Gordon-Maxwell Equations}, Adv. Nonlinear Stud. {\bf 4} (2004), 307--322.

\bibitem{DM2} 
 T. D’Aprile, D. Mugnai, {\em Solitary waves for nonlinear Klein-Gordon-Maxwell and Schr\"odinger-Maxwell equations}, Proc. Roy. Soc. Edinburgh
Sect. A, 134, (2004), 893--906.

\bibitem{DP}
P. d'Avenia, L. Pisani, {\em Nonlinear Klein-Gordon equations coupled with Born-Infeld type equations}, Electron. J. Differential Equations 2002, no. 26, 13 pp.

\bibitem{DPS09}
P. d'Avenia, L. Pisani, G. Siciliano, {\em Dirichlet and Neumann problems for Klein-Gordon-Maxwell systems}, Nonlinear Analysis {\bf 71} (2009), e1985--e1995.

\bibitem{DPS10}
P. d'Avenia, L. Pisani, G. Siciliano, {\em Klein-Gordon-Maxwell systems in a bounded domain}, Discrete Contin. Dyn. Syst. {\bf 26} (2010), 135--149. 


\bibitem{DHV}
O. Druet, E. Hebey, J. V\'etois, {\em Static Klein–Gordon–Maxwell–Proca systems in $4$-dimensional closed manifolds II}, to appear on J. Reine Angew. Math., DOI: \href{http://dx.doi.org/10.1515/crelle-2013-0125}{10.1515/crelle-2013-0125}.

\bibitem{F}
B. Felsager, {\em Geometry, particles, and fields}, Graduate Texts in Contemporary Physics, Springer-Verlag, New York, 1998.

%

\bibitem{GMP}
M. Ghimenti, A.M. Micheletti, A. Pistoia, {\em The role of the scalar curvature in some singularly perturbed coupled elliptic systems on Riemannian manifolds}, Discrete Contin. Dyn. Syst. {\bf 34} (2014), 2535--2560. 


\bibitem{GN}
A.S. Goldhaber, M.M. Nieto, {\em Photon and Graviton mass limits}, Reviews of Modern Physics {\bf 82} (2010), 939--979.
%

\bibitem{HT}
E. Hebey, T.T. Truong, {\em Static Klein-Gordon-Maxwell-Proca systems in $4$-dimensional closed manifolds}, J. Reine Angew. Math. {\bf 667} (2012), 221--248.

\bibitem{HW} E. Hebey, J. Wei, {\em Resonant states for the static Klein-Gordon-Maxwell-Proca system}, Math. Res. Lett. {\bf 19} (2012),  953--967.
%

\bibitem{Lions842}
P.-L. Lions,  {\em The concentration-compactness principle in the calculus of variations. The locally compact case, part 2}, Ann. Inst. H. Poincar\'e Anal. Non Lin\'eaire {\bf 1} (1984), 223--283. 


%


\bibitem{N1}
G.L. Naber, {\em Topology, geometry, and gauge fields. Foundations}, Texts in Applied Mathematics {\bf 25}, Springer-Verlag, New York, 2011.

\bibitem{N2}
G.L. Naber, {\em Topology, geometry, and gauge fields. Interactions}, Texts in Applied Mathematics {\bf 141}, Springer-Verlag, New York, 2011.


\bibitem{Proca1}
A. Proca, {\em Sur la Th\'eorie du Positron}, C. R. Acad. Sci. Paris {\bf 202} (1936), 1366--1368.

\bibitem{Proca2}
A. Proca, {\em Sur la Th\'eorie Ondulatoire des \'Electrons Positifs et N\'egatifs}, J. Physique Radium Ser. VII {\bf 7} (1936), 347--353.

\bibitem{Proca3}
A. Proca, {\em Sur les Photons et les Particules charge pure}, C. R. Acad. Sci. Paris {\bf 203} (1936), 709--711.

\bibitem{Proca4}
A. Proca, {\em Particles Libres: Photons et Particules `charge pure'}, J. Physique Radium Ser. VII {\bf 8} (1937), 23--28.

\bibitem{Proca5}
A. Proca, {\em Th\'eorie Non Relativiste des Particles a Spin Entier}, J. Physique Radium Ser. VII {\bf 9} (1938), 61--66.


\bibitem{SzulkinWeth} 
A. Szulkin, T. Weth, {\em Ground state solutions for some indefinite variational problems}, J. Funct. Anal. {\bf 257} (2009), 3802--3822. 



\bibitem{Wang} F. Wang, {\em Ground-state solutions for the electrostatic nonlinear Klein-Gordon-Maxwell system}, Nonlinear Anal. {\bf 74} (2011), 4796--4803.

\bibitem{Willem} 
M. Willem, {\em Minimax Theorems}, Birkhäuser Verlag 1996.

\bibitem{yu}
Y. Yu, {\em Solitary waves for nonlinear Klein-Gordon equations coupled with Born-Infeld theory}, Ann. Inst. H. Poincar\'e Anal. Non Lin\'eaire {\bf 27} (2010), 351--376.

\end{thebibliography}
\end{document}